\documentclass[11pt,oneside,reqno]{amsart}
\usepackage{amssymb}
\usepackage{amsmath}
\usepackage{amscd}
\usepackage{amsfonts}
\usepackage{mathrsfs}
\usepackage{stmaryrd}
\usepackage{tikz}
\usetikzlibrary{arrows}
\usepackage[T1]{fontenc}
\allowdisplaybreaks

\makeatletter
\DeclareFontFamily{U}{tipa}{}
\DeclareFontShape{U}{tipa}{m}{n}{<->tipa10}{}
\newcommand{\arc@char}{{\usefont{U}{tipa}{m}{n}\symbol{62}}}%

\newcommand{\arc}[1]{\mathpalette\arc@arc{#1}}

\newcommand{\arc@arc}[2]{%
  \sbox0{$\m@th#1#2$}%
  \vbox{
    \hbox{\resizebox{\wd0}{\height}{\arc@char}}
    \nointerlineskip
    \box0
  }%
}
\makeatother

\newcommand{\nord}{\mbox{\scriptsize ${\circ\atop\circ}$}}

\usepackage{bbm}

\theoremstyle{definition}
\newtheorem{theorem}{Theorem}[section]

\newtheorem{thm}[theorem]{Theorem}
\newtheorem{prop}[theorem]{Proposition}

\newtheorem{defn}[theorem]{Definition}
\newtheorem{lemma}[theorem]{Lemma}
\newtheorem{cor}[theorem]{Corollary}
\newtheorem{prop-def}{Proposition-Definition}[section]
\newtheorem{rema}[theorem]{Remark}

\newtheorem{nota}[theorem]{Notation}

\newcommand{\N}{{\mathbb N}}
\newcommand{\C}{{\mathbb C}}
\newcommand{\Z}{{\mathbb Z}}

\newcommand{\Hom}{\textrm{Hom}}

\newcommand{\one}{\mathbf{1}}
\renewcommand{\d}{\mathbf{d}}
\newcommand{\wt}{\text{wt}}

\newcommand{\Res}{\text{Res}}

\newcommand{\g}{{\mathfrak g}}
\newcommand{\h}{{\mathfrak h}}
\newcommand{\n}{{\mathfrak n}}

\begin{document}

\setlength{\oddsidemargin}{0cm} \setlength{\evensidemargin}{0cm}
\baselineskip=18pt

\title[First Cohomologies of Affine, Virasoro, and Lattice VOA]{First cohomologies of affine, Virasoro and lattice vertex operator algebras}
\author{Fei Qi}

\begin{abstract}
    In this paper we study the first cohomologies for the following three examples of vertex operator algebras: (i) the simple affine VOA associated with a simple Lie algebra with positive integral level; (ii) the Virasoro VOA corresponding to minimal models; and (iii) the lattice VOA assoicated with a positive definite even lattice. We prove that in all these cases, the first cohomology $H^1(V, W)$ consists of zero-mode derivations for every $\N$-graded $V$-module $W$ (where the grading is not necessarily given by the $L(0)$ operator). This agrees with the conjecture made by Yi-Zhi Huang and the author in 2018. The relationship between the first cohomology of the VOA and that of the associated Zhu's algebra is also discussed. 
    
    
 \end{abstract}

\maketitle

\section{Introduction}

In the representation theory of various algebras (including, but not limited to, commutative associative algebras, associative algebras, Lie algebras, etc.), one of the main tools is the cohomological method. The powerful method of homological algebra often provides a unified treatment of many results in representation theory, giving not only solutions to open problems, but also conceptual understandings to the results.

Vertex operator algebras (VOAs hereafter) arose naturally in both mathematics and physics (see \cite{BPZ}, \cite{B}, and \cite{FLM}) and are analogous to both Lie algebras and commutative associative algebras. In \cite{H-coh}, Yi-Zhi Huang introduced a cohomology theory for grading-restricted vertex algebras and modules that are analogous to Harrison cohomology for commutative associative algebras. Huang's construction uses linear maps from the tensor powers of the vertex algebra to suitable spaces of ``rational functions valued in the algebraic completion of a $V$-module'' satisfying some natural conditions, including a technical convergence condition. Geometrically, this cohomology theory is consistent with the geometric and operadic formulation of VOA. Algebraically, the first and second cohomologies are given by (grading-preserving) derivations and first-order deformations, similarly to those for commutative associative algebras. 

The vanishing theorem of an associative algebra states that the algebra is semisimple if and only if first Hochschild cohomologies with coefficients in all bimodules vanish, or equivalently, if and only if every derivation from the algebra to any bimodule is inner. Naturally, we expect similar results in VOAs. In \cite{Q-Coh}, the author generalized Huang's work and introduced a cohomology theory for meromorphic open-string vertex algebras (a non-commutative generalization to VOAs) and their bimodules, analogously to the Hochschild cohomology for (not-necessarily-commutative) associative algebras and the bimodules. The early version of \cite{HQ-Coh-reduct} proved a vertex algebraic version of the if part of the vanishing theorem: if $V$ is a meromorphic open-string vertex algebra such that the first cohomology $H^1(V, W)= 0$ for every $V$-bimodule $W$, then every left $V$-module satisfying a technical but natural convergence condition is completely reducible. 

For cohomological methods, the only if part of the vanishing theorem is much more useful in calculations. A common practice is to start from a long exact sequence, then use the vanishing theorem to fill zeros in various spots, and conclude various isomorphisms. While $H^1(V, W)=0$ holds trivially for every irreducible $V$-module $W$ with non-integral lowest weight, we found that if the lowest weight is an integer, $H^1(V, W)$ is generally not zero. Indeed for a VOA $V$ and a $V$-module $W$, it follows from the Jacobi identity that the zero-mode $v\mapsto w_0 v = \Res_x e^{xL(-1)}Y_{W}(v, -x)w$ of any weight-1 element $w$ defines a grading-preserving derivation $V\to W$ (provided that the map is not constantly zero), called zero-mode derivations. We denote the space of such derivations by $Z^1(V, W)$ (which stays zero if $W$ has no elements of integral weights). Conceptually, we believe that such derivations exist because the cohomologies in \cite{H-coh} and in \cite{Q-Coh} are analogous to Harrison and Hochschild for (commutative) associative algebras that ``missed'' the analogous Lie algebra structure on a VOA. The final version of \cite{HQ-Coh-reduct} proved the same conclusion under the revised condition that $H^1(V, W) = Z^1(V, W)$ for every $V$-bimodule $W$. 

A natural question then arises along the way: does $H^1(V, W) = Z^1(V, W)$ hold for every $V$-bimodule? We conjecture that this holds  when $V$ is strongly rational, i.e., if every weak $V$-module is a direct sum of its irreducible submodules. Since the associated Zhu's algebra is semisimple in this case, it seems promising to study the relation of $H^1(V, W)$ to the first cohomology of Zhu's algebras. We attempted and disappointedly found that though Zhu's algebra can simplify the computation and provide some interesting results regarding the first cohomology, it is usually insufficient to fully determine $H^1(V, W)$, partly because Zhu's algebra forgets some crucial structures. For example, Zhu's algebra associated with the lattice VOA is one-dimensional when the lattice is even, positive definite, and unimodular. In this case, Zhu's algebra tells nothing about $H^1(V, W)$.

Nevertheless, we still managed to prove $H^1(V, W)=Z^1(V, W)$ for some of the most important examples of VOAs $V$ and $V$-modules $W$. More precisely, we consider the cases where $V$ takes:
\begin{enumerate}
    \item the simple affine VOA $V=L_{\hat\g}(l, 0)$ assoicated with a simple Lie algebra $\g$ with positive integral level $l$;
    \item the Virasoro VOA with central charge 
    $$c = 1 - \frac{(p-q)^2}{6pq} \text{ for some integers }p, q \geq 2, gcd(p, q)=1.$$
    \item the lattice VOA $V_{L_0}$ assoicated with a positive definite even lattice $L_0$.
\end{enumerate}
For every $V$-module $W$ with $\N$-grading, we showed that $H^1(V, W) = Z^1(V, W)$ in all these three cases. In particular, the conclusion applies to modules for affine and lattice VOAs with $L(0)$-grading, since the $L(0)$-weights are all nonnegative. On the other hand, modules for the Virasoro VOA may have negative $L(0)$-weights. We managed to prove $H^1(V, W)=Z^1(V, W)$ when the lowest $L(0)$-weight of $W$ is greater or equal to $-3$. Though we believe the same should hold in general and proposed an algorithmic approach that can further lower the bound $-3$, it is unlikely that this approach leads to a general proof. More conceptual approaches are necessary and shall be considered in the future. 

We focus on the $\N$-grading and the $L(0)$-grading because they are the most natural choices. The $\N$-grading case is considered to be more important because the conclusions in \cite{HQ-Coh-reduct} only requires $H^1(V, W) = Z^1(V, W)$ to hold for every $V$-bimodule with the canonical $\N$-grading. We believe that the conjecture should hold for arbitrary grading, and every (not necessarily grading-preserving) derivation should be a ``zero-mode derivation'' associated with some element $w\in W$. But the method in the current paper is unlikely to achieve that goal.

We emphasize that all the $V$-modules $W$ considered in this paper are regarded as $V$-bimodules via the right action $w\otimes v \mapsto Y_{WV}^W(w, x)v = e^{xL(-1)}Y_W(v, -x)w$. We further emphasize that the conclusions in this paper are by-no-means sufficient to conclude the conjecture made in \cite{HQ-Coh-reduct}. The conjecture requires $H^1(V, W) = Z^1(V, W)$ for every $V$-bimodule $W$, the right-action on which is not necessarily related to the left-action. To fully prove the conjecture, it is necessary to classify all $V$-bimodules. This shall be discussed in future studies.

The paper is organized as follows. 

Section 2 reviews the basic definitions and discusses some general facts regarding derivations and zero-mode derivations, including the construction of a natural map from $H^1(V, W)$ to the first cohomology of the (higher level) Zhu's algebra assoicated with $V$.

Sections 3, 4 and 5 study $H^1(V, W)$ for the three types of VOAs $V$ and their irreducible modules $W$. We first show $H^1(V, W)=Z^1(V, W)$ for $L(0)$-graded $W$ with nonnegative weights, then for $W$ with the canonical $\N$-grading with zero lowest weight, finally for $W$ with arbitrary $\N$-grading. In addition, Section 3 also discusses the image of the map constructed in Section 2 (Proposition \ref{Affine-Zhu}); Section 4 gives a counter-example to $H^1(V,W) = Z^1(V, W)$ when $V$ is the Virasoro VOA not corresponding to minimal models (Proposition \ref{L(-1)w-der}); Section 4 also discusses the case when $L(0)$-weight of $W$ is negative (Theorem \ref{Vir-neg-energy-thm}). 

Section 6 provides a summary of the results and gives some remarks regarding future work.

\noindent\textbf{Acknowledgements.} The author would like to thank Yi-Zhi Huang for his long-term constant support, especially for the valuable suggestions and comments that greatly improved this paper. The author would also like to thank the Pacific Institute of Mathematical Science for offering the opportunity of teaching a network-wide graduate course on vertex algebras and Ethan Armitage, Ben Garbuz, Zach Goldthorpe, Nicholas Lai, and Mihai Marian for their participation. Many ideas for this paper arose during lectures and discussions.

\section{Basic definitions and some general facts}

In this paper, we assume that the reader is familiar with the basic notions and results in the theory of VOAs. In particular,
we assume that the reader is familiar with weak modules, $\N$-gradable weak
modules, generalized modules, and related results. Our terminology and
conventions follow those in \cite{FLM}, \cite{FHL}, and \cite{LL}. 

In this section, $(V, Y, \one, \omega)$ is VOA; $(W, Y_W)$ is a generalized $V$-module, i.e., $W = \coprod_{m\in \C} W_{[m]}$ is the direct sum of generalized eigenspaces 
$$W_{[m]} = \{w\in W: (L(0)- m)^n w = 0 \text{ for some }n\in \Z_+\}$$
of the operator $L(0)$. We also require that for every $m\in \C$, $W_{[m-n]}=0$ for every sufficiently large $n\in \Z_+$. In case $W$ is irreducible, $W = \coprod_{m\in \gamma + \N} W_{[m]} = \coprod_{n=0}^\infty W_{[\gamma+n]}$ for some $\gamma\in \C$. In other words, $W$ is $\N$-gradable by the operator $L(0)_s - \gamma$, where $L(0)_s$ is the semisimple part of the operator $L(0)$. This is one of the many cases when the grading of $W$ is better shifted by a constant away from the $L(0)$-grading. 

\noindent \textbf{Conventions of Notations and Terminologies  }In this paper, we will use the notation $\d_W$ for the grading operator of $W$. Certainly it is most natural to choose $\d_W = L(0)_s$. But we would like to allow $\d_W$ on each irreducible component of $W$ to be shifted by a constant $\alpha\in \C$, i.e., $d_W$ takes $L(0)_s + \alpha$ on an irreducible component. If $W$ is irreducible and has lowest weight $\gamma$, $\d_W = L(0)_s - \gamma$ gives an $\N$-grading, called the canonical $\N$-grading. From now on, $W_{[m]}$ will be used to denote the eigenspace of $\d_W$ instead of $L(0)_s$, i.e., for $m\in \C$, 
$$W_{[m]}= \{w\in W: \d_W w = m w\}. $$
We also say a $V$-module $W$ is $L(0)$-isomorphic (or $L(0)_s$-isomorphic) to $V$, if the $L(0)$-graded (or $L(0)_s$-graded) module $W$ is isomorphic to $V$. Note that we are not requiring $\d_W = L(0)$ or $\d_W=L(0)_s$. 


\subsection{First cohomology of VOAs} 

The cohomology theory of grading-restricted vertex algebras was introduced in \cite{H-coh} using certain maps from tensor products of $V$ to the space of $\overline{W}$-valued rational functions. The maps should satisfy $L(0)$- and $L(-1)$-conjugation properties and a natural but technical composable condition. Theorem 1.1 in \cite{H-1st-sec-coh} points out that the first cohomology $H^1(V, W)$ is isomorphic to the space of \textit{grading-preserving} derivations. Since this paper mainly concerns the first cohomology, we will use this result without going into the technical composable condition here. 

\begin{defn}
A linear map $F: V \to W$ is a grading-preserving derivation, if 
\begin{enumerate}
    \item $
    F\circ L(0) = \d_W \circ F$. 
    \item For every $u, v\in V$, 
    $$F(Y(u, x)v) = Y_W(u, x) F(v) + Y_{WV}^W(F(u), x)v. $$
\end{enumerate}
Here $Y_{WV}^W(w, x)v = e^{xL(-1)}Y_W(v, -x)w$. 
\end{defn}

\begin{nota}
$Y_{WV}^W$ is indeed an intertwining operator of type $\binom{W}{WV}$ satisfying the corresponding Jacobi identity. We use the notation $w_n$ for the coefficient of $x^{-n-1}$ in $Y_{WV}^W(w,x)$, as we normally do with intertwining operators. More precisely, $w_n: V \to W$ is a linear map sending $v\in V$ to $\Res_x x^n Y_{WV}^W(w, x)v$. For convenience, we also write
$$Y_{WV}^W(w, x) = \sum_{n\in \Z} w_n x^{-n-1} \in \Hom(V, W)[[x, x^{-1}]]. $$
By a straightforward computation, we see that
$$w_i v = (-1)^{i+1} v_i w + (-1)^{i+2} L(-1)v_{i+1}w + \cdots + \frac{(-1)^{i+1+j}}{j!}L(-1)^j v_{i+j}w + \cdots, $$
for $w\in W, v\in V, n\in \Z$. This formula will be frequently used in computing $w_i v$ without explicit reference. 
\end{nota}

\begin{rema}
It is shown in \cite{H-1st-sec-coh} that the first cohomology $H^1(V, W)$ and $\widehat{H}^1(V, W)$ constructed in \cite{H-coh} are both isomorphic to the space of grading-preserving derivations from $V$ to $W$. Note that \cite{H-coh} constructed the cohomology theory for grading-restricted vertex algebras and modules, where the conformal element is not required to exist. In particular, the $L(0)$-operator in \cite{H-coh} is not required to be a component of the vertex operator associated with the conformal element $\omega$. It can still be defined even without the existence of conformal element. The notation $L(0)$ is replaced by $\d_W$ in the later papers \cite{Q-Coh} and \cite{HQ-Coh-reduct}. 
\end{rema}

\begin{rema}
We should emphasize that $H^1(V, W)$ depends on the choice of the grading $\d_W$. Different choices of grading give different cohomologies. For example, if $W$ is an irreducible $V$-module whose lowest $L(0)$-weight $\gamma$ is not an integer, it is clear that $H^1(V, W) = 0$ if $\d_W = L(0)$, while $H^1(V, W)$ is generally nonzero with the canonical $\N$-grading given by $\d_W = L(0)-\gamma$. 
\end{rema}

\begin{rema}
It also follows trivially from the definition that 
$$H^1(V, W_1\oplus W_2) = H^1(V, W_1)\oplus H^1(V, W_2). $$
Thus if every $V$-module is completely reducible, it suffices to study $H^1(V, W)$ for irreducible $W$. 
\end{rema}

\begin{rema}
It follows from the definition of $F$ that $F\circ L(-1) = L(-1) \circ F$. The following argument is given by Huang: By setting $u=v=\one$ one sees immediately that  $F(\one)=0$. Thus
$$F(Y(v, x)\one)=Y_{WV}^W(F(v), x)\one.$$
Taking the derivative on both sides and use $L(-1)$-derivative property for $Y_{WV}^W$. 
$$\frac d {dx} F(Y(v, x)\one)=\frac d {dx} Y_{WV}^W(F(v), x)\one=Y_{WV}^W(L(-1)F(v), x)\one.$$
On the other hand, it follows from the $L(-1)$-derivative property of $Y$ that
$$\frac{d}{dx} F(Y(v, x)\one)=F(Y(L(-1)v, x)\one)=Y_{WV}^W(F(L(-1)v), x)\one.$$
Thus $Y_{WV}^W(F(L(-1)v), x)\one=Y_{WV}^W(L(-1)F(v), x)\one$. Taking the limit $x\to 0$, we obtain $F(L(-1)v)=L(-1)F(v)$.
\end{rema}

\begin{defn}
A derivation is called zero-mode derivation if there exists some homogeneous $w\in W$ of weight 1, such that 
$$F(v) = w_0 v$$
for every $v\in V$. Here 
$w_0$ is the coefficient of $x^{-1}$ in the series $Y_{WV}^W(w, x)v$. The space of zero-mode derivations is denoted by $Z^1(V, W)$. 
\end{defn}

\begin{rema}
It is proved in \cite{HQ-Coh-reduct} that if $V$ is a (meromorphic open-string) vertex algebra such that for every $V$-bimodule $(W, Y_W^L, Y_W^R)$, $\widehat{H}^1(V, W)=Z^1(V, W)$, then any left $V$-module satisfying a natural composability condition is a direct sum of irreducible submodules. It is also conjectured in \cite{HQ-Coh-reduct} that if the every weak $V$-module is a direct sum of irreducible submodules, then for every $V$-bimodule $W$, $\widehat{H}^1(V, W) = Z^1(V, W)$. Note that in this conjecture, the right $V$-module structure $Y_W^R$ is not necessarily given by $Y_{WV}^W$. 
\end{rema}

\subsection{Some general observations}

\begin{prop}
The dimension of the space of zero-mode derivations satisfies
$$\dim Z^1(V, W) \leq \dim W_{[1]}/ L(-1)W_{[0]}.$$ 
Equality holds when $W$ is irreducible, not $L(0)_s$-isomorphic to $V$ and $\d_W = L(0)_s + \alpha$ with $\alpha\neq 0$. \end{prop}

\begin{proof}
The first conclusion follows from the $L(-1)$-derivative property of $Y_{WV}^W$: if $w=L(-1)u$ for some $u\in W_{[0]}$, then 
$$w_0 v = \Res_{x} Y_{WV}^W (L(-1)u, x)v = \Res_x \frac d {dx} Y_{WV}^W(u, x)v,$$
which is zero since the coefficient of $x^{-1}$ vanishes in any derivative of any formal series.

To show the equality holds under the additional assumption, we start from the assumption $w_0v = 0$ for every $v\in V$. Then in particular, 
$$0 = w_0\omega = -\omega_0 w + L(-1)\omega_1 w + L(-1)^2 u = L(-1)(-\alpha w + L(-1)u)$$
for some $u\in W_{[0]}$. Since $W$ is not $L(0)$-isomorphic to $V$, $\ker L(-1) = 0$ (see \cite{Li-vacuum-like}). Since $\alpha \neq 0$, we have $w=L(-1)u/\alpha \in L(-1)W_{[0]}$. 
\end{proof}

\begin{rema}
See Remark \ref{strictly-less-1} for a counter-example to the equality when $W$ is $L(0)$-isomorphic to $V$. See Remark \ref{strictly-less-2} for a counter-example to the equality when $\alpha=0$. 
\end{rema}

\begin{prop}\label{Der-gen}
Let $V$ be a VOA generated by a subset $S$. Let $W$ be a $V$-module and $F: V\to W$ be a derivation. If $F$ sends every element in $S$ to zero, then $F$ is a zero map.
\end{prop}

\begin{proof}
This follows directly from an induction, using the facts that $V$ is spanned by 
$$a^{(1)}_{n_1}\cdots a^{(m)}_{n_m} \one, a^{(1)}, ..., a^{(m)}\in S,  n_1, ..., n_m\in \Z, $$
and 
\begin{align*}
F(a^{(i)}_{n_i} v) &= \Res_x x^{n_i} F(Y_V(a^{(i)}, x)v)\\
&= \Res_x x^{n_i} \left(Y_W(a^{(i)}, x)F(v) + Y_{WV}^W(F(a^{(i)}), x)v \right)\\
&= a^{(i)}_{n_i}F(v) + F(a^{(i)})_{n_i} v,
\end{align*}
for any $a^{(i)}\in S$, $v\in V$. 
\end{proof}

\begin{rema}\label{Rmk-2-12}
The proposition shows that there exists \textit{at most} one derivation sending the generators of $V$ to some designated elements in $W$. Thus to classify the derivations, it suffices to focus on the images on the generators. Whether or not these images are compatible is an interesting question. In Proposition \ref{L(-1)w-der} we will see a case study of the compatibility problem. 
\end{rema}



\subsection{Zhu's algebra and its bimodules}

Zhu defined an associative algebra $A(V)$ associated with a VOA $V$ in \cite{Zhu-Modular}, called the Zhu's algebra. Dong, Li, and Mason gave a generalization $A_N(V)$ in \cite{DLM-higher-Zhu}, called the level-$N$ Zhu's algebra, or higher level Zhu's algebra. We recall the definition in \cite{DLM-higher-Zhu} here. 

\begin{defn}
For any fixed natural number $N\in \N$,  $A_N(V) = V/O_N(V)$ where $O_N(V)$ is spanned by elements of the form 
\begin{align}
    \Res_x x^{-2N-2-n}Y((1+x)^{L(0)+N}u, x)v \label{ucircv}
\end{align}
with $u, v\in V, n\in \N$, and elements 
$$L(0) v + L(-1)v $$
with $v\in V$. For $u, v\in A_N(V)$, a binary operation $u*v$ is defined by
$$u *_N v = \sum_{m=0}^N (-1)^m\binom{m+N}N\Res_x x^{-N-m-1}Y((1+x)^{L(0)+N}u, x)v.$$
When $N=0$, $A_0(V)$ coincides with $A(V)$ defined in \cite{Zhu-Modular}.
\end{defn}

\begin{thm}[\cite{Zhu-Modular}, \cite{DLM-higher-Zhu}]
$(A_N(V), *_N)$ forms an associative algebra. \end{thm}

Given a $V$-module $W$, Frenkel and Zhu defined a $A(V)$-bimodule $A(W)$ in \cite{Frenkel-Zhu}. Huang and Yang generalized the construction and defined an $A_N(V)$-bimodule $A_N(W)$ in \cite{HY-lio-aa}. We recall the definition in \cite{HY-lio-aa} here. 

\begin{defn}
Let $A_N(W) = W/O_N(W)$ where $O_N(W)$ is spanned by elements of the form 
$$\Res_x x^{-2N-2} Y_W((1+x)^{L(0)+N}v, x)w$$
where $v\in V, w\in W$, and elements of the form $$\d_W w+L(-1)w$$
where $w\in W$. For $v\in V, w\in W$, define
\begin{align*}
    v *_N w &= \sum_{m=0}^N \binom{m+N}N \Res_x x^{-1-m-N} Y_W((1+x)^{L(0)+N}v, x)w,\\
    w *_N v &= \sum_{m=0}^N \binom{m+N}N \Res_x x^{-1-m-N} Y_{WV}^W((1+x)^{\d_W+N}w, x)v.
\end{align*}
Here $Y_{WV}^W$ is the intertwining operator defined as follows
\begin{align*}
    & Y_{WV}^W:  W \otimes V\to W[[x,x^{-1}]]\\
    & Y_{WV}^W(w,x)v = e^{xL(-1)}Y_W(v,-x)w = \sum_{n\in \Z} w_n v x^{-n-1}. 
\end{align*}
\end{defn}

\begin{thm}[\cite{HY-lio-aa}]
$(A_N(W), *)$ forms an $A_N(V)$-bimodule. 
\end{thm}

\begin{rema}
\begin{enumerate}
    \item It should be emphasized that in case $N=0$, $A_0(V) = A(V)$ but $A_0(W)$ is \textit{different} to the bimodule $A(W)$ constructed in \cite{Frenkel-Zhu}, where $(\d_W+L(-1))w$ is not included in the relation of $O(W)$. In order to emphasize the difference, we will stick to the notation $A_0(V)$ and $A_0(W)$ in this paper, in spite that $A_0(V) = A(V)$. 
    \item It should also be noted that if $N>0$, the elements $L(0)v + L(-1)v$ are generally not included in the subspace spanned by elements of form (\ref{ucircv}). To guarantee that $A_N(V)$ acts by zero-modes, it is necessary to further require $L(0)v +L(-1)v\in O_N(V)$. Please see \cite{AB} for further discussions. 
    \item It should be noted that recently, the authors of \cite{HY-lio-aa} published a corrigendum \cite{HY-lio-aa-correct} and posted a new version of \cite{HY-lio-aa} on arXiv. In the new version, the requirement that $L_W(0)w + L_W(-1)w\in O_N(W)$ was removed. The right action is also modified so that $L_W(0)+L_W(-1)$ does not necessarily act trivially on $A_N(W)$. The resulted bimodule $A_N(W)$ in \cite{HY-lio-aa} will be larger than what we use in the current paper. Their revision does not impact the validity of any results here. 
\end{enumerate}

\end{rema}

The following results from \cite{HY-lio-aa} will be useful in our argument:

\begin{lemma}\label{HY-Lemma}
\begin{enumerate}
    \item For every $v\in V$
$v*_N O_N(W)\subseteq O_N(W), O_N(w)*_N v \subseteq O_N(W)$, 
    \item \label{HY-Lemma-1}(Lemma 4.4, \cite{HY-lio-aa}, Part 2) For every $v\in V, w\in W$ and integers $p\geq q \geq 0$, 
    $$\Res_x x^{-2N-2-p}Y_W((1+x)^{L(0)+N+q}u, x)w\in O_N(W), $$
    $$\Res_x x^{-2N-2-p}Y_{WV}^W((1+x)^{\d_W+N+q}u, x)w\in O_N(W).$$
    In particular, for every $n\in \N$, 
    $$\Res_x x^{-2N-2-n}Y_{WV}^W((1+x)^{\d_W+N}w, x)v\in O_N(W).$$
    \item (Lemma 4.4, \cite{HY-lio-aa}, Part 3) \label{HY-Lemma-2} For every $v\in V, w\in W$, $$v*_Nw - w*_N v = -\Res_x Y_{WV}^W((1+x)^{\d_W-1}w, x)v,$$
    $$v*_Nw - w*_N v = -\Res_x Y_W((1+x)^{L(0)-1}v, x)w$$
in $A_N(W)$. 

\end{enumerate}
\end{lemma}

\subsection{First cohomology of an associative algebra} 

The cohomology theory of associative algebras is given by Hochschild in \cite{Hochschild}.

\begin{defn}
Let $(A, *)$ be an associative algebra, $B$ be an $A$-bimodule. A linear map $f:A\to B$ is a derivation, if for any $u, v\in A$, $$f(u*v) = u* f(v) + f(u)*v.$$
We say $f$ is an inner derivation, if there exists some $w\in B$ such that 
$$f(v) = v*w - w*v. $$
The first cohomology $H^1(A, B)$ is defined as the quotient of the space of derivations modulo the subspace of inner derivations. 
\end{defn}

\begin{thm}
$A$ is semisimple if and only if $H^1(A, B)=0$ for every $A$-bimodule $B$. In other words, every derivation is inner.
\end{thm}

\begin{rema}
We should note that the proof of the only-if-part in \cite{Hochschild} requires $A$ to be finite-dimensional but does not require $B$ to be finite-dimensional. 
\end{rema}

\begin{rema}
We did not mention inner derivations for vertex algebras, because in our current context, all such inner derivations are automatically zero. In general, if $V$ as a meromorphic open-string vertex algebra and $W$ is a bimodule, then an inner derivation $F$ is given by the formula
$$F(v) = Y_W^L(v, x)w - e^{xL(-1)}Y_W^R(w, -x)v.$$
But in our current situation, $Y_W^R$ is given by $Y_{WV}^W$. Thus $F = 0$. 
\end{rema}


\subsection{The map $H^1(V, W) \to H^1(A_N(V), A_N(W))$}
\begin{prop}\label{Prop-2-23}
Let $F: V\to W$ be a derivation from the VOA $V$ to the $V$-module $W$. For any $N\in \N$, we define $f: A_N(V) \to A_N(W)$ by
$$f(v) = F(v) + O_N(W). $$
Then $f$ is a derivation from the associative algebra $A_N(V)$ to the $A_N(V)$-bimodule $A_N(W)$. The map $F\mapsto f$ gives a linear map from $H^1(V, W)\to H^1(A_N(V), A_N(W))$.
\end{prop}

\begin{proof}
The composition $v\mapsto F(v) \mapsto F(v)+O_N(W)$ gives a well-defined map from $V$ to $A_N(W)$. We show that $F$ maps $O_N(V)$ to $O_N(W)$. Then the composition factorizes through $A_N(V)= V/O_N(V)$, inducing $f: A_N(V) \to A_N(W)$. 

By definition of $F$, for every $v\in V$
$$F(L(-1)v + L(0)v) = L(-1) F(v) + \d_W F(v).$$
So $F$ maps $L(-1)v + L(0)v$ to $O_N(W)$. 

Also by definition of $F$, for every $u, v\in V$, we have
\begin{equation}\label{def-der}
    F(Y(u, x)v)=Y_W(u, x)f(v) + Y_{WV}^W(f(u), x)v.
\end{equation}
Without loss of generality, assume $u$ is homogeneous and let $n\in \N$, we take $\Res_x x^{-2N-2-n}(1+x)^{\text{wt}(u)+N}$ on both sides of (\ref{def-der}), we have 
\begin{align*}
    & F\left(\Res_x x^{-2N-2-n}Y((1+x)^{\text{wt}(u)+N}u, x)v\right)\\
    & \quad = \Res_x x^{-2N-2-n}Y_W((1+x)^{\text{wt}(u)+N}u, x)F(v) + \Res_x x^{-2}Y_{WV}^W((1+x)^{\text{wt}(u)+N}F(u), x)v
\end{align*}
Note that since $\d_W F(u) = F(L(0) u) = \text{wt}(u) F(u)$, we have $(1+x)^{\text{wt}(u)}=(1+x)^{\d_W}F(u)$. Thus what we have indeed is
\begin{align}
    & F\left(\Res_x x^{-2N-2-n}Y((1+x)^{L(0)+N}u, x)v\right)\nonumber \\
     & =  \Res_x x^{-2N-2-n}Y_W((1+x)^{L(0)+N}u, x)F(v) \label{Prop-2-23-1}\\
      & \quad + \Res_x x^{-2N-2-n}Y_{WV}^W((1+x)^{\d_W+N}F(u), x)v. \label{Prop-2-23-2}
\end{align}
This equality obviously extends to nonhomogeneous $u\in V$. 
By definition of $O_N(W)$, (\ref{Prop-2-23-1}) falls in $O_N(W)$. By Lemma 4.4 Part (2) in \cite{HY-lio-aa} (recorded as Lemma \ref{HY-Lemma} (\ref{HY-Lemma-1})) here), (\ref{Prop-2-23-2}) is also in $O_N(W)$. Thus $F$ maps every element of the spanning set in $O_N(V)$ into $O_N(W)$. 

Now we show that 
\begin{align}
    f(u*_Nv) = u*_N f(v)+f(u)*_N v.\label{Prop-2-23-3}
\end{align}
We simply apply the operator $\sum_{m=0}^N (-1)^m \binom{m+N} N \Res_x x^{-N-m-1} (1+x)^{\wt(u)+N}$ to both sides of (\ref{def-der}). Using the fact that $F((1+x)^{L(0)} u) = (1+x)^{\d_W} F(u) = (1+x)^{\wt(u)} F(u)$, we have
\begin{align*}
 & \sum_{m=0}^N (-1)^m \binom{m+N}N \Res_x x^{-N-m-1} F\left(
 Y_V((1+x)^{L(0)+N}u, x)v)\right) \\
 = &  
\sum_{m=0}^N (-1)^m \binom{m+N}N \Res_x  x^{-N-m-1} \left( Y_W((1+x)^{L(0)+N}u, x)F(v) +  Y_{WV}^W ((1+x)^{\d_W+N} F(u), x)v\right), 
\end{align*}
which is precisely (\ref{Prop-2-23-3}). 
\end{proof}

\begin{cor}\label{Corollary-1-22}
Let $F:V\to W$ be a derivation that sits in the  kernel of the map $H^1(V, W)\to H^1(A_N(V), A_N(W))$. Then there exists some $w\in W$ such that 
$$F(v) = -\Res_x Y_W((1+x)^{L(0)-1}v, x) w + O(W). $$
\end{cor}

\begin{proof}
Since $v\mapsto F(v)+O(W)$ is inner, we know that 
$$F(v) = v*_N w - w*_N v + O(W) = -\Res_x Y_W((1+x)^{L(0)-1}v, x)w + O(W)$$
where the last equality follows from Lemma \ref{HY-Lemma} (\ref{HY-Lemma-2}). 
\end{proof}

\begin{rema}\label{Rmk-1-23}
In case $A_N(V)$ is semisimple,  $H^1(A_N(V), A_N(W)) = 0$. We once hoped to use it to classify $H^1(V, W)$. But it turns out that this result helps only when $W=V$ and when $A(V)$ is not too trivial, as will be addressed in later sections. 
\end{rema}

\begin{rema}
Possibly, the result might be helpful when Zhu's algebra is not semisimple. One natural problem is to determine whether or not the map $H^1(V, W)\to H^1(A_N(V), A_N(W))$ is surjective, which is interesting and shall be studied in the future.  
\end{rema}

\section{First cohomologies of affine VOAs}\label{Section-2}

In this section, we will use the Whitehead lemma for simple Lie algebras to study the first cohomologies of the affine VOA associated with a simple Lie algebra $\g$.

\subsection{Derivations of Lie algebras and Whitehead lemma} Let $\mathfrak{g}$ be a Lie algebra and $M$ be a $\mathfrak{g}$-module. A linear map $f: \g \to M$ is a derivation if 
$$f([x,y])=x f(y) - y f(x). $$
$f$ is called an inner derivation, if there exists $m\in M$ such that
$$f(x) = x m.$$
\begin{lemma}[Whitehead lemma]
If $\mathfrak{g}$ is simple, then for every $\g$-module $M$, every derivation $f: \g \to M$ is inner (see \cite{Sam}).
\end{lemma}

\subsection{Affine VOA assoicated with simple Lie algebras}

The affine VOA assoicated with a simple Lie algebra was first constructed in \cite{Frenkel-Zhu}. Here we give a brief review following \cite{LL}. 

Let $\g$ be a finite-dimensional simple Lie algebra. Let $\Phi = \Phi_+ \cup \Phi_-$ be the root system, $\Phi_+$ (resp. $\Phi_-$) be the set of positive roots (resp. negative) roots. Let $\mathfrak h$ be a Cartan subalgebra. Denote the triangular decomposition of $\g$ by 
$$\g = \n_+ \oplus \h \oplus \n_-. $$
where $\n_\pm = \coprod_{\alpha\in \Phi_\pm}\g_{\alpha}$. For every $\alpha\in \Phi_+$, we denote by $h_\alpha$ the unique vector such that $[\g_\alpha, \g_{-\alpha}] = \C h_\alpha$ and $\alpha(h_\alpha)=2$.  

For $\lambda\in \mathfrak{h}^*$, we consider the Verma module, i.e., the induced $\g$-module from the one-dimensional $(\n_+\oplus \h)$-module where $\n_+$ acts trivially and $h\in \h$ acts by the scalar $\lambda(h)$. Recall that $\lambda$ is dominant integral, if $$\lambda(h_\alpha) =  \frac{2\langle\lambda,\alpha\rangle}{\langle\alpha,\alpha\rangle}\in \N$$ for every $\alpha\in \Phi_+$, here $\langle\cdot, \cdot\rangle$ is the normalized Killing form on $\h^*$ such that $\langle\alpha, \alpha\rangle = 2$ for every long root $\alpha$. If $\lambda$ is dominant integral, the Verma module has a unique irreducible quotient. We denote the irreducible quotient by $M_\lambda$. 

Now we consider the affine Lie algebra
$$\hat{\g} = \g \otimes_\C \C[t, t^{-1}] \oplus \C k$$
with the Lie bracket 
\begin{align*}
    [a\otimes t^m, b\otimes t^n] &= [a,b]\otimes t^{m+n} + m\langle a,b\rangle\delta_{m+n,0}k, \\
    [k, g\otimes t^m] &= 0,
\end{align*}
where $a,b\in \g, m,n\in \Z$, $\langle\cdot, \cdot\rangle$ is the normalized Killing form on $\mathfrak{g}$, $k$ is the central element. For $l\in \C$, let $V_{\hat\g}(l, \lambda)$ be the induced $\hat{g}$-module from the $(\g\otimes t\C[t] \oplus \g \oplus \C k)$-module $M_\lambda$, where $\g\otimes t\C[t]$ acts trivially, $k$ acts by the scalar $l$.

It follows from Section 6.2 and Section 6.6 in \cite{LL} and \cite{DLM-Regular} that
\begin{enumerate}
    \item $V_{\hat{\g}}(l, 0)$ is a VOA that has a simple quotient $L_{\hat\g}(l, 0)$. 
    \item $V_{\hat{\g}}(l, \lambda)$ has a unique quotient $L_{\hat\g}(l, \lambda)$ that is an irreducible $V_{\hat\g}(l, 0)$-module. 
    \item If the level $l$ is a positive integer, then every irreducible $L_{\hat\g}(l, 0)$-modules is isomorphic to $L_{\hat\g}(l, \lambda)$ for some dominant integral $\lambda\in \h^*$ satisfying $\lambda( h_{\theta})\leq l$, where $\theta$ is the highest root of $\g$. 
    \item If the level $l$ is a positive integer, then every weak $L_{\hat\g}(l, 0)$-module is a direct sum of irreducible $L_{\hat\g}(l,0)$-modules. 
\end{enumerate}

\subsection{Image in $H^1(A_0(V), A_0(W))$}

Let $l\in \C$ and $V=L_{\hat\g}(l, 0)$. Let $W$ be any generalized $V$-module. The grading operator $\d_W$ of $W$ does not necessarily coincides with the semisimple part of $L(0)$. Let $A_0(V)$ be the (level-0) Zhu's algebra assoicated with $V$ and $A_0(W)$ be the $A_0(V)$-bimodule assoicated with $W$. We emphasize that $A_0(W)$ we used here is subject to the additional relation $(\d_W + L(-1))w \in O_0(W)$ compared to $A(W)$ in \cite{Frenkel-Zhu}. 

\begin{prop}\label{Affine-Zhu}
Let $F: V \to W$ be any derivation. Then there exists an element $w_{(1)}\in W_{[1]}$, such that 
$$F(v) \equiv (w_{(1)})_0 v \mod O_0(W).$$
\end{prop}

\begin{proof}
We first recall some basic facts. 

\begin{enumerate}
    \item $V_{(1)} \simeq \g$. The map $a(-1)\one\otimes b(-1)\one \mapsto a(0)b(-1)\one$ coincides with the Lie algebra structure on $V_{(1)}$. 
    \item $W_{[1]}$ is a $\g$-module with the following action.  
    $$a\in \g, a\cdot w = a(0)w.$$
    \item $W_{[1]}\cap O_0(W)$ is a $\g$-submodule in $W_{[1]}$. This follows from the fact that $w\in O_0(W)\Rightarrow \Res_x Y_W((1+x)^{L(0)-1}a(-1)\one, x)w = a(0)w \in O_0(W)$.
    \item The image $W_{[1]} + O_0(W)$ of $W_{[1]}$ in $A_0(W)$ is then a $\g$-module. This essentially follows from the isomorphism $W_{[1]} + O_0(W)/ O_0(W)\simeq W_{[1]} / W_{[1]}\cap O_0(W)$. 
\end{enumerate}

For any fixed $a\in \g$, $F(a(-1)\one)\in W_{[1]}$. We write $\theta(a) \in W_{[1]}+O_0(W)$ as the image of $F(a(-1)\one)$ in $A_0(W)$. Then the map $a \mapsto \theta(a)$ gives a map from $\g$ to the $\g$-module $W_{[1]}+O_0(W)$. 

From the fact that $F$ is a derivation, for any $a,b\in \g$, 
$$F(a(0)b(-1)\one) = a(0) F(b(-1)\one) + F(a(-1)\one)_0 b(-1)\one.$$
Note that $a(0)b(-1)\one = [a,b](-1)\one$, So the image of the left-hand-side in $W_{[1]}+O_0(W)$ is precisely $\theta([a,b])$. The image of first term on the right-hand-side is $a(0)\theta(b)$. For the second term on the right-hand-side, we compute as follows
\begin{align}
    F(a(-1)\one)_0 b(-1)\one &= \Res_x Y_{WV}^W(\theta(a), x)b(-1)\one \nonumber \\
    &= \Res_x e^{xL(-1)} Y_W(b(-1)\one,-x)\theta(a) \nonumber \\
    &\equiv \Res_x (1+x)^{-\d_W}Y_W(b(-1)\one,-x)\theta(a) \mod O_0(W) \label{3-2-line-3}\\
    &= \Res_x Y_W\left((1+x)^{-L(0)}b(-1)\one,-\frac{x}{1+x} \right) (1+x)^{-\d_W}\theta(a)\nonumber\\
    &= \Res_y Y_W\left((1+y)^{L(0)}b(-1)\one,y \right) (1+y)^{\d_W}\theta(a) \cdot (-(1+y)^{-2})\label{3-2-line-5}\\
    &= -\Res_y Y_W(b(-1)\one, y) \theta(a)\nonumber\\
    &= -b(0)\theta(a),\nonumber
\end{align}
where (\ref{3-2-line-3}) follows from the formula proved in Lemma 2.5 of \cite{H-aa-rep-va}
$$e^{xL(-1)}(1+x)^{\d_W} = (1+x)^{\d_W + L(-1)} \equiv 1 \mod O_0(W);$$
(\ref{3-2-line-5}) follows from the change of variable formula
$$\Res_x f(x) = \Res_y f(g(y))g'(y),$$
with $g(y)=-y/(1+y)$ (see \cite{Zhu-Modular}). 
Thus we have shown that 
$$\theta([a,b]) \equiv a(0)\theta(b) - b(0)\theta(a) \mod O_0(W). $$
This is to say that the map $\theta: \g \to W_{[1]}+O_0(W)$ is a Lie algebra derivation. From Whitehead Lemma, there exists an element $w\in W_{[1]}+O_0(W)$ such that 
$$\theta(a) = a\cdot w = a(0)w. $$
This is to say that 
$$F(a(-1)\one) = a(0)w + O_0(W).$$
The conclusion then follows with the choice $w_{(1)} = -w$ and the fact that $a(0)w \equiv -w_0 a \mod O_0(W)$ (cf. Lemma \ref{HY-Lemma}). 
\end{proof}

\begin{rema}
The conclusion also holds for $V = V_{\hat\g}(l,0)$. Note also that $W$ is not necessarily graded by $L(0)$. Also note that Proposition \ref{Affine-Zhu} does not have any restriction of the level $l$. 
\end{rema}

\begin{rema}
The result only states that the image of $H^1(V, W)$ is zero in $H^1(A_0(V), A_0(W))$. It is certainly insufficient to conclude that $H^1(A_0(V), A_0(W)) = 0$. 
\end{rema}

\subsection{The module $L_{\hat\g}(l, \lambda)$ $(l\in \Z_+)$ and $L(0)$-grading}

Let $l\in \Z_+$ and $\lambda\in \h^*$ be dominant integral, satisfying $\lambda(h_\theta) \leq l$ for the highest root $\theta$. Let $V = L_{\hat\g}(l, 0)$ and $W = L_{\hat\g}(l, \lambda)$. We will show that $H^1(V, W) = Z^1(V, W)$.

\begin{lemma}\label{lowest-weight}
\begin{enumerate}
    \item The lowest $L(0)$-weight of $L_{\hat\g}(l, \lambda)$ is nonnegative. The lowest $L(0)$-weight is zero only when $\lambda = 0$. 
    \item If $W$ is not isomorphic to $V$ and $W_{[1]} \neq 0$, then $W_{[1]}$ is the lowest weight subspace. 
\end{enumerate}
\end{lemma}

\begin{proof}
Let $\{u^{(i)}, i=1, ..., d\}$ be an orthonormal basis with respect to the normalized Killing form. Then 
$$\sum_{i=1}^d u^{(i)}\otimes u^{(i)}$$
is the Casimir element in the universal enveloping algebra $U(\g)$. It is well-known that the Casimir element acts on $M_\lambda$ by the scalar 
$$\langle \lambda, \lambda\rangle + 2\langle\lambda, \rho\rangle$$
where $\rho$ is half of the sum of the positive roots. Since $\lambda$ is dominant integral, both $\langle\lambda, \lambda\rangle$ and $\langle\lambda, \rho\rangle$ are nonnegative. The Casimir element acts by zero only when $\lambda = 0$. Let $h$ be the dual Coxeter number of $\g$. From the Sugawara construction, $L(0)$ acts on the lowest weight subspace by
$$\frac 1 {2(l+h)} \sum_{i=1}^d u^{(i)}(0) u^{(i)}(0) = \frac 1 {2(l+h)} \left(\langle \lambda, \lambda\rangle + 2\langle\lambda, \rho\rangle\right),$$
a scalar that is nonnegative. The scalar is zero only when $\lambda = 0$. This concludes (1). 

For (2), we start by noticing that if $W_{[1]}\neq 0$, then $W$ has to be $\Z$-graded. So from (1), the lowest $L(0)$-weight is nonnegative. If the lowest $L(0)$-weight is zero, then $\lambda = 0$ and $W = L_{\hat\g}(l, 0) = V$, contradicting the assumption that $W$ is not isomorphic to $V$. Thus the lowest $L(0)$-weight has to be at least 1. The conclusion follows from the assumption that $W_{[1]}\neq 0$. 
\end{proof}

\begin{thm}\label{affine}
Let $l\in \Z_+$, $V = L_{\hat\g}(l, 0)$ and $W=L_{\hat\g}(l, \lambda)$, where $\lambda$ is a dominant integral weight satisfying $\langle \lambda, h_\theta\rangle \leq l$. We take $\d_W=L(0)$. Then for every derivation $F: V\to W$, $F(v) = (w_{(1)})_0 v$ for some $w_{(1)}\in W_{[1]}$. Thus, $H^1(V, W) = Z^1(V, W)$. 
\end{thm}

\begin{proof}
Let $\gamma$ be the lowest $L(0)$-weight. We know from Lemma \ref{lowest-weight} that $\gamma\geq 0$. Thus it suffices to study the following cases.  
\begin{enumerate}
    \item If $\gamma\notin \Z$, then $F=0$. We simply take $w_{(1)}=0$. 
    \item If $\gamma \geq 2$, we know that $F(a(-1)\one) = 0$ for every $a\in \g$, which implies that $F=0$.
    \item If $\gamma=1$, then we have
    $$F((a(-1)\one)_0 b(-1)\one = -b(0)F(a(-1)\one).$$
    So 
    $$F(a(0)b(-1)\one) = F([a,b](-1)\one) =  a(0)F(b(-1)\one) - b(0) F(a(-1)\one). $$
    This means that the map 
    $$a \mapsto F(a(-1)\one)$$ 
    is a Lie algebra derivation from $\g$ to $W_{[1]}$. By Whitehead lemma, there exists $w_{(1)}\in W_{[1]}$ such that 
    $$F(a(-1)\one) = -a(0) w_{(1)} = (w_{(1)})_0 a(-1)\one.$$
    Thus 
    $$F(v) = (w_{(1)})_0 v$$
    is a zero-mode derivation. 
    \item If $\gamma=0$, we know from Lemma \ref{lowest-weight} that $W=V$. We know from the Proposition \ref{Affine-Zhu} that there exists $w_{(1)}\in V_{(1)}$, such that
    $$F(v) \equiv (w_{(1)})_0 v \mod O_0(V).$$
    So the map $g(v) = F(v) - (w_{(1)})_0 v$ is a derivation with image in $O_0(V)$. In particular, for every $a\in \g$, $g(a(-1)\one)$ is a homogeneous element in $O_0(V)$ of weight 1. However, from the fact that $A_0(V) = U(\g)/\langle e_\theta^{l+1}\rangle$ (see \cite{Frenkel-Zhu}), it is clear that $O_0(V)$ contains no homogeneous element of weight 1. Thus $g(a(-1)\one)= 0$ and $F(a(-1)\one) = (w_{(1)})_0 a(-1)\one$. This implies that $F(v) = (w_{(1)})_0 v$.
\end{enumerate} 
\end{proof}

\begin{rema}
The $\gamma=0$ case has also been known in \cite{HQ-Coh-reduct}. Zhu's algebra provides an alternative proof. But for $\gamma=1$ case, Zhu's algebra will not help unless we know there exists $V$-modules $W_2$ and $W_3$ of the same lowest weight, such that the fusion rule $N_{WW_2}^{W_3} \neq 0$ (cf. Remark \ref{Rmk-1-23}). This assumption usually does not hold. In case it holds, then from the conclusion of the previous proposition, there exists some $w_{(1)}\in W_{[1]}$ such that $F(a(-1)\one) - (w_{(1)})_0 a(-1)\one$ is a homogeneous element of weight $h$ in $O(W)$, which we denote by $\theta$. Assume that $\theta\neq 0$. Then since for any $a\in \g$, $a(0)\theta\in O(W)$, we know that $W_{[1]} \subseteq O(W)$ and thus $A_0(W) = 0$. But the assumption $N_{WW_2}^{W_3}\neq 0$ means $A_0(W)$ cannot be zero. So we have a contradiction. 
\end{rema}

\subsection{The module $L_{\g}(l, \lambda)$ $(l\in \Z_+)$ with $\N$-grading} In this subsection we consider $W=L_{\hat\g}(l, \lambda)$ with $\N$-grading. The main effort focuses on the canonical $\N$-grading given by $\d_W = L(0) - (\langle\lambda, \lambda\rangle + 2 \langle\lambda, \rho\rangle) / 2(l+h)$. 

One should note that with a grading different from the $L(0)$ one, the choice of derivations is different. Thus, $H^1(V, W)$ and $Z^1(V, W)$ are different from those with $L(0)$-gradings. 

\begin{thm}\label{affine-canonical-N-grading}
Let $V = L_{\hat{\g}}(l, 0)$ with $l\in \Z_+$. Let $\lambda\in \h^*$ be dominant integral satisfying $\lambda(h_\alpha)\leq l$. Let $W = L_{\hat\g}(l, \lambda)$ with the canonical $\N$-grading given by the operator $\d_W = L(0)-(\langle\lambda, \lambda\rangle + 2 \langle\lambda, \rho\rangle) / 2(l+h)$. Then $H^1(V, W) = Z^1(V, W)$. \end{thm}

\begin{proof}
If $\lambda = 0$ then $W = V$. The canonical $\N$-grading coincides with the $L(0)$-grading and thus has been taken care of in Theorem \ref{affine}. Thus we focus on the case $\lambda\neq 0$. 

Let $F: V\to W$ be a derivation. Then for every $a, b\in \g$, 
\begin{align*}
    F([a,b](-1)\one) &=F(a(0) b(-1)\one) = a(0) F(b(-1)\one) + F(a(-1)\one)_0 b(-1)\one\\
    &= a(0) F(b(-1)\one)-b(0)F(a(-1)\one) + L(-1)b(1)F(a(-1)\one)\\
    &\in a(0) F(b(-1)\one)-b(0)F(a(-1)\one) + L(-1)W_{[0]}.
\end{align*}
This motivates us to consider the space $W_{[1]} / L(-1)W_{[0]}$. 

View $W_{[0]}$ and $W_{[1]}$ as $\g$-modules, where the actions are given by $g(0), g\in \g$. We first show that $L(-1): W_{[0]}\to W_{[1]}$ is an injective homomorphism of $\g$-modules. Indeed, it follows from $L(-1)$-derivative property that $(L(-1)v)_0 = 0$ for every $v\in V$. It follows from $L(-1)$-commutator property that
\begin{align*}
    L(-1)g(0) &= g(0)L(-1) + [L(-1), g(0)] \\
        &= g(0)L(-1) + (L(-1)g(-1)\one)_0 =g(0)L(-1).
\end{align*}
Thus $L(-1)$ a homomorphism. Since $\lambda\neq 0$, $W$ is not $L(0)$-isomorphic to $V$ and thus has no vacuum-like vectors. Thus $\ker L(-1)= 0$ and $L(-1)$ is injective. 

Thus, $W_{[1]}/L(-1)W_{[0]}$ is a $\g$-module. Consider the restriction of $F$ on $V_{(1)} \simeq \g$, whose image is then in $W_{[1]}$. Let $\bar F: \g \to W_{[1]}/L(-1)W_{[0]}$ be the composition of the restriction and the canonical projection. Then $\bar F$ satisfies
$$\bar F([a,b](-1)\one) = a(0) \bar F(b(-1)\one) - b(0) \bar F(a(-1)\one) $$
and thus forms a Lie algebra derivation. Whitehead lemma implies that
\begin{align*}
    \bar F(a(-1)\one) & = F(a(-1)\one) + L(-1)W_{[0]} \\
    &=  a(0) (w_{[1]}+L(-1)W_{[0]}) + L(-1)W_{[0]} \in W_{[1]}/ L(-1)W_{[0]}. 
\end{align*}
for some $w_{[1]}\in W_{[1]}$. Since $a(0)$ commutes with $L(-1)$, we thus have
$$F(a(-1)\one) \in a(0) w_{[1]} + L(-1) W_{[0
]}.$$
Consider now the map $F_1: V\to W$ defined by 
$$F_1(v) = F(v) + (w_{[1]})_0 v. $$
Notice that $(w_{[1]})_0 a(-1) = -a(0) w_{[1]} + L(-1)a(1) w_{[1]} \in -a(0)w_{[1]} + L(-1)W_{[0]}$, thus 
$$F_1(a(-1)\one) = F(a(-1)\one) - a(0) w_{[1]} + L(-1)a(1)w_{[1]}\in L(-1)W_{[0]}. $$
This in particular implies that $F_1(a(-1)\one)_0 = 0$ from $L(-1)$-derivative property. Thus for $a,b\in \g$, we have
$$F_1(a(0) b(-1)\one) = a(0) F_1(b(-1)\one),$$
which means that $F_1$ is a $\g$-module homomorphism from $\g$ to $L(-1)W_{[0]}$. 

However, since $L(-1)$ is an injective homomorphism, $L(-1)W_{[0]}$ is isomorphic to $W_{[0]}$ which is $M_\lambda$. From Schur's lemma, if $M_\lambda$ is not isomorphic to the adjoint $\g$-module, the the map $a(-1)\one\mapsto F_1(a(-1)\one)$ is zero. So $F_1$ is a derivation sending every generator of $V$ to zero. Thus $F_1(v) = 0$ and $F(v) + (w_{[1]})_0 v = 0$ for every $v\in $V. Therefore, $F$ is a zero-mode derivation. 

It remains to consider the case when $M_\lambda\simeq \g$. In this case, let $\psi: \g \to M_\lambda$ be a $\g$-module isomorphism (unique up to a scalar). Then for every $a\in \g$, we have
$$F_1(a(-1)\one) = L(-1)\psi(a).$$
We will proceed to show that $F_1$ is also a zero-mode derivation. 

For each $\alpha\in \Phi$, let $t_\alpha$ be an element in $\h$ satisfying
$$\alpha(h) = \langle t_\alpha, h\rangle, h\in \h. $$
Also pick $\{e_\alpha: \alpha\in \Phi\}$ to form a Chevalley basis, i.e., each $e_\alpha$ is a root vector in $\g_\alpha$; 
\begin{align*}
    [e_\alpha, e_{-\alpha}] & = \langle e_\alpha, e_{-\alpha}\rangle t_\alpha = \frac{2}{\langle \alpha, \alpha\rangle} t_\alpha; 
\end{align*}
and for $\alpha, \beta\in \Phi$ with $\alpha+\beta\neq 0$, 
\begin{align*}
    [e_\alpha, e_\beta] &= c_{\alpha\beta} e_{\alpha+\beta}, 
\end{align*}
where the coefficients $c_{\alpha\beta}$ satisfies $c_{\alpha\beta} = -c_{\beta\alpha} = -c_{-\alpha, -\beta}$, together with the following property: if $\alpha, \beta, \gamma\in \Phi$ satisfies $\alpha+\beta+\gamma=0$, then 
\begin{align}
    \frac{c_{\alpha\beta}}{\langle \gamma, \gamma\rangle} = \frac{c_{\beta\gamma}}{\langle \alpha, \alpha\rangle} = \frac{c_{\gamma\alpha}}{\langle \beta, \beta\rangle}\label{cab-formula}
\end{align}
(See \cite{Hum} and \cite{Sam} for details). Let $\alpha_1, ..., \alpha_r\in \Phi^+$ be the simple positive roots. We denote the elements $t_{\alpha_i}$ simply by $t_i$. Let $t_1^\vee, ..., t_r^\vee \in \h$ satisfying $\langle t_i, t_j^\vee\rangle = \alpha_i(t_j^\vee)= \delta_{ij}, i, j= 1, ..., r$. Consider now the element $w_{(1)}\in W_{[1]}$ of the following form 
$$w_{(1)} = \sum_{i=1}^r t_i(-1)\psi(t_i^\vee) + \sum_{\alpha\in \Phi} \frac{\langle \alpha, \alpha\rangle}{2} e_\alpha(-1)\psi(e_{-\alpha}). $$
Informally, $w_{(1)}$ is the Casimir element
$$\Omega = \sum_{i=1}^r t_i t_i^\vee + \sum_{\alpha\in \Phi} \frac{\langle \alpha, \alpha\rangle}{2}e_\alpha e_{-\alpha}$$
twisted by $\psi$ and the negative-one-mode. 

We now show that $w_{(1)}$ spans a trivial $\g$-submodule of $W_{[1]}$. We first note that for every $h\in \h$, 
\begin{align*}
    h(0)w_{(1)} &= \sum_{i=1}^r [h,t_i](-1)\psi(t_i^\vee)+t_i(-1)\psi([h,t_i^\vee]) \\
    & \quad+ \sum_{\alpha\in \Phi}\frac{\langle\alpha,\alpha\rangle}2 \left([h,e_\alpha](-1)\psi(e_{-\alpha})+ e_\alpha(-1) \psi([h,e_{-\alpha}]\right)\\
    &= 0+\sum_{\alpha\in \Phi}\frac{\langle\alpha,\alpha\rangle}2 \left(\alpha(h)e_\alpha(-1)\psi(e_{-\alpha})-\alpha(h) e_\alpha(-1) \psi(e_{-\alpha}\right)=0. 
\end{align*}
Thus $w_{(1)}$ is of weight zero. 

For every $j=1, ..., r$, 
\begin{align}
    e_{\alpha_j}(0)w_{(1)} &= \sum_{i=1}^r [e_{\alpha_j},t_i](-1)\psi(t_i^\vee)+t_i(-1)\psi([e_{\alpha_j},t_i^\vee]) \label{Line1}\\
    & \quad+ \sum_{\alpha\in \Phi}\frac{\langle\alpha,\alpha\rangle}2 \left([e_{\alpha_j},e_\alpha](-1)\psi(e_{-\alpha})+ e_\alpha(-1) \psi([e_{\alpha_j},e_{-\alpha}]\right)\label{Line2}
\end{align}
By $[e_{\alpha}, t_\beta] = -\alpha(t_\beta)e_\alpha = -\langle t_\alpha, t_\beta\rangle e_\alpha$, and $h = \sum_{i=1}^r \langle h, t_i\rangle t_i^\vee = \sum_{i=1}^r \langle h, t_i^\vee\rangle t_i$, (\ref{Line1}) can be simplified as
$$ -e_{\alpha_j}(-1)\psi(t_j) - t_j(-1) \psi(e_{\alpha_j}), $$
which cancels out with the $\alpha = -\alpha_j$ summand in (\ref{Line2}). 

For other summands in (\ref{Line2}), we separate $\Phi$ into a disjoint union of $\alpha_j$-strings and show that the summation along each $\alpha_j$-string yields zero. Fix some $\alpha_j$-string consisting of positive roots and let $\beta$ be lowest element, so that the $\alpha_j$-string is of the form 
\begin{align}
    \beta, \beta+\alpha_j, ..., \beta+q\alpha_j \label{3-8-summand-index}
\end{align}
is the $\alpha_j$-string. We start with the summand corresponding to $\beta$, namely
\begin{align}
    \frac{\langle \beta, \beta\rangle}2 c_{\alpha_j\beta}e_{\beta+\alpha_j}(-1)\psi(e_{-\beta}) \label{3-8-line-1}
\end{align}
(second half is zero because $[e_{\alpha_j}, e_{-\beta}]=0$) together with the summand corresponding to $\beta+\alpha_j$ 
\begin{align}
    & \frac{\langle \beta+\alpha_j, \beta+\alpha_j\rangle}2 c_{\alpha_j, \beta+\alpha_j}e_{\beta+2\alpha_j}(-1)\psi(e_{-\beta-\alpha_j}) \label{3-8-line-2}\\
    & + \frac{\langle \beta+\alpha_j, \beta+\alpha_j\rangle}2  c_{\alpha_j, -\beta-\alpha_j}e_{\beta+\alpha_j}(-1) \psi(e_{-\beta}). \label{3-8-line-3}
\end{align}
From Formula (\ref{cab-formula}) and skew-symmetry of $c_{\alpha\beta}$, we have 
$$ \frac{c_{\alpha_j\beta}}{\langle \alpha_j+\beta, \alpha_j+\beta\rangle} = \frac{c_{-\beta-\alpha_j, \alpha_j}}{\langle \beta, \beta\rangle} = -\frac{c_{\alpha_j, -\beta-\alpha_j}}{\langle \beta, \beta\rangle} $$
Thus(\ref{3-8-line-1}) and (\ref{3-8-line-3}) add up to zero. So the summation of $\alpha=\beta$ summand and $\alpha=\beta+\alpha_j$ summand is simply (\ref{3-8-line-2}). If we further add up with the summand assoicated with $\alpha=\beta+2\alpha_j$, namely, 
\begin{align}
    & \frac{\langle\beta+2\alpha_j,\beta+2\alpha_j\rangle}2 c_{\alpha_j, \beta+2\alpha_j} e_{\beta+3\alpha_j}(-1)\psi(e_{-\beta-2\alpha_j})\label{3-8-line-4}\\
    & + \frac{\langle\beta+2\alpha_j,\beta+2\alpha_j\rangle}2 c_{\alpha_j, -\beta-2\alpha_j}e_{\beta+2\alpha_j}(-1) \psi(e_{-\beta-\alpha_j})\label{3-8-line-5}
\end{align}
the sum is simply (\ref{3-8-line-4}) since (\ref{3-8-line-3}) and (\ref{3-8-line-5}) add up to zero. So what we have is a telescoping sum. Repeating the process with summands assoicated with $\beta+3\alpha_j$, ..., $\beta+(q-1)\alpha_j$, we see that the sum is precisely  
\begin{align}
    & \frac{\langle\beta+(q-1)\alpha_j,\beta+(q-1)\alpha_j\rangle}2 c_{\alpha_j, \beta+(q-1)\alpha_j} e_{\beta+q\alpha_j}(-1)\psi(e_{-\beta-(q-1)\alpha_j})\label{3-8-line-6}. 
\end{align}
The last summand associated with $\beta+q\alpha_j$ only has one term
\begin{align}
        & \frac{\langle\beta+q\alpha_j,\beta+q\alpha_j\rangle}2 c_{\alpha_j, -\beta-q\alpha_j}e_{\beta+q\alpha_j}(-1) \psi(e_{-\beta-(q-1)\alpha_j})\label{3-8-line-7}
\end{align}
since $\beta+(q+1)\alpha_j\notin\Phi$. By the same relation, (\ref{3-8-line-6}) and (\ref{3-8-line-7}) add up to zero. So the partial sum from (\ref{Line2}) along every $\alpha_j$-string is zero. Thus we showed that for every positive simple root $\alpha_i$, $e_{\alpha_i}(0)w_{(1)} = 0$. This means that $w_{(1)}$ is also a highest weight vector. Thus it generates a trivial $\g$-module. 

Now we can compute the zero-mode derivation defined by $w_{(1)}$: for any $a\in \g$, 
\begin{align*}
    (w_{(1)})_0 a(-1)\one = -a(0) w_{(1)} + L(-1)a(1)w_{(1)}. 
\end{align*}
We have seen above that first term is zero. To compute the second term, we first notice that
\begin{align*}
    a(1)w_{(1)} &= \sum_{i=1}^r a(1) t_i(-1)\psi(t_i^\vee) + \sum_{\alpha\in \Phi} \frac{\langle \alpha, \alpha\rangle}{2}a(1)e_\alpha(-1)\psi(e_{-\alpha})\\
    &= \sum_{i=1}^r ([a,t_i](0) + l\langle a, t_i\rangle )\psi(t_i^\vee) + \sum_{\alpha\in \Phi} \frac{\langle \alpha, \alpha\rangle}{2}([a, e_\alpha](0)+l\langle a, e_\alpha\rangle)\psi(e_{-\alpha})\\
    &= \sum_{i=1}^r \psi([[a,t_i], t_i^\vee]) + \sum_{\alpha\in \Phi} \frac{\langle \alpha, \alpha\rangle}{2} \psi([[a,e_\alpha], e_{-\alpha}]) \\
    &\quad + l\psi\left(\sum_{i=1}^r\langle a, t_i\rangle t_i^\vee + \sum_{\alpha\in \Phi} \langle a, e_\alpha\rangle \frac{\langle \alpha, \alpha\rangle} 2 e_{-\alpha}\right)\\
    &= (2h+l)\psi(a). 
\end{align*}
Thus, 
$$(w_{(1)})_0 a(-1)\one = L(-1)a(1)w_{(1)} = (2h+l)L(-1)\psi(a). $$ 
If we let 
$$F_2(v) = F_1(v) - \left(\frac 1 {2h+l}w_{(1)}\right)_0 v, $$ then $F_2(a(-1)\one) = 0$ for every $a\in \g$. This shows that $F_2(v)=0$ for every $v\in V$. This is to say that for every $v\in V$, 
$$0 = F(v) + (w_{[1]})_0 v - \left(\frac{1}{2h+l}w_{(1)}\right)_0 v.$$
Therefore, $F$ is a zero-mode derivation. 
\end{proof}

\begin{thm}
Let $V = L_{\hat{\g}}(l, 0)$ with $l\in \Z_+$. Let $W = L_{\hat\g}(l, \lambda)$ with arbitrary $\N$-grading. Then $H^1(V, W) = Z^1(V, W)$. \end{thm}

\begin{proof}
Let $\Xi\in \N$ be the lowest weight of $L_{\hat\g}(l, \lambda)$. The $\Xi=0$ case is proved in Theorem \ref{affine-canonical-N-grading}. The remaining cases can be similarly handled as in Part (2) and (3) of Theorem \ref{affine} and shall not be repeated here. 
\end{proof}

\begin{rema}\label{strictly-less-1}
When $\lambda=0$ and the lowest weight of $W=L_{\hat\g}(l, 0)$ is 1, it is clear that $Z^1(V, W) = 0$ (since the lowest weight vector is vacuum-like whose zero-mode is zero). So $\dim Z^1(V, W) < W_{[1]}/L(-1)W_{[0]}$.
\end{rema}


\section{First cohomologies of Virasoro VOAs}\label{Section-3}

In this section, we will study the first cohomology of the Virasoro VOA corresponding to the minimal models. 

\subsection{Virasoro VOA and modules} The Virasoro VOA is first constructed in \cite{Frenkel-Zhu}. Here we give a brief review following \cite{LL}. 

Let $Vir$ be the Virasoro algebra, i.e., $Vir = 
\bigoplus_{n\in \Z} \C L_n \oplus \C \textbf{c}$, with 
$$[c, Vir] = 0; [L_m, L_n] = (m-n)L_{m+n}+ \delta_{m+n, 0} \frac{m^3 - m}{12}\textbf{c}.$$
For $c, h\in \C$, let $\C \one_{c, h}$ be a one-dimensional vector space, on which $L_m$ acts trivially for every $m>0$, $L_0$ acts by the scalar $h$, and $\textbf{c}$ acts by the scalar $c$. Let $M(c, h)$ be the induced $Vir$-module from the $(\bigoplus_{n\geq 0} \C L_n \oplus \C \textbf{c})$-module $\C \one_{c, h}$. 

It follows from Section 5.5 and 6.1 of \cite{LL} that
\begin{enumerate}
    \item $M(c, 0)$ has a quotient $V(c, 0)$ that forms a VOA, on which $L(-1)\one_{c, 0} = 0$. We will abbreviate $\one_{c, 0}$ as $\one$. 
    \item $M(c, h)$ has a quotient $L(c, h)$ that forms an irreducible $V(c, 0)$-module.
    \item $V(c, 0)$ has a unique simple quotient $L(c, 0)$.  
\end{enumerate}

It follows from \cite{Wang} and \cite{DLM-Regular} that
\begin{enumerate}
    \setcounter{enumi}{3}
    \item If \begin{align}
        c= c_{p,q}= 1 - \frac{6(p-q)^2}{pq}\label{Formula-5}
    \end{align}
    for some integers $p, q\geq 2$ mutually prime, then there are only finitely many irreducible $L(c, 0)$-modules. 
    \item If $c = c_{p,q}$ as in (\ref{Formula-5}), then every irreducible $L(c, 0)$-module is isomorphic to $L(c, h)$ where 
    \begin{align}
        h = h_{m,n} = \frac{(np-mq)^2-(p-q)^2}{4pq} \label{Formula-6}
    \end{align}
    for some integers $m,n\in \Z_+$ such that $m<p$ and $n<q$. 
    \item If $c = c_{p,q}$ as in (\ref{Formula-5}), then every weak $L(c, 0)$-module is a direct sum of irreducible $L(c, 0)$-modules.
\end{enumerate}


\begin{rema}
Let $V = L(c, 0)$ for some $c=c_{p,q}$ as in (\ref{Formula-5}), $W= L(c, h)$ for some $h= h_{m,n}$ as in (\ref{Formula-6}). Since $A_0(V)$ is semisimple, for every derivation $F: V\to W$, there exists $w\in W$, such that
\begin{align}
    F(\omega) &\equiv \omega*w - w*\omega \nonumber\\
    &= (L(0)+L(-1))w \label{Rema-4-2-1}\\
    &= -\alpha w \mod O_0(W)  \label{Rema-4-2-2} 
\end{align}
where (\ref{Rema-4-2-1}) follows from Lemma \ref{HY-Lemma} (\ref{HY-Lemma-2}); (\ref{Rema-4-2-2}) follows from $\d_W = L(0)+\alpha$ and $(\d_W + L(-1))w\in O_0(W)$. In case $W$ is graded by $L(0)$, we have 
$$F(\omega)\in O_0(W).$$
However, this fact is not useful when $W
\neq V$, simply because $A_0(W)\neq 0$ only in some very limited cases. It can provide some but very minor simplifications in the computations of $H^1(V, V)$, as will be seen in the discussions below. 
\end{rema}

\subsection{Irreducible module with nonnegative lowest weight}

\begin{thm}\label{Vir-pos-energy}
Let $V = L(c, 0)$ for some $c=c_{p,q}$ as in (\ref{Formula-5}), $W= L(c, h)$ for some $h= h_{m,n}$ as in (\ref{Formula-6}) with $\d_W = L(0)$. Then $H^1(V, W) = 0$ for $h\geq 0$. 
\end{thm}

\begin{proof}
If $h\notin \Z$, then $F(\omega)=0$. If $h>2$, then since $F(\omega)$ is of weight 2, it is necessary that $F(\omega) = 0$ We show that $h=1$ is impossible when $c=c_{pq}$. In fact, if $h=1$, then 
\begin{align*}
    (np-mq)^2 - (p-q)^2 = 4pq & \Rightarrow (np-mq - p-q)(np-mq+p+q)=0\\
    & \Rightarrow (n-1)p = (m+1)q \text{ or } (n+1)p = (m-1)q
\end{align*}
where $m,n,p,q$ are integers  satisfying $p,q>0, (p, q) = 1, 0<m<p, 0<n < q$. Since $(p,q)=1$, it is necessary that $p|m+1$ or $p|m-1$, neither of which is possible. 

Thus the only remaining cases are $h=2$ and $h=0$. 
\begin{enumerate}
    \item If $h=2$, then $F(\omega)$ is of lowest weight 2. It follows from $F(L(-1)\omega) = L(-1)F(\omega)$ that 
    \begin{align*}
        F(\omega)_0 \omega =0 & \Rightarrow -\omega_0 F(\omega) + L(-1)\omega_1 F(\omega)=0 \Rightarrow -L(-1)F(\omega) + 2L(-1)F(\omega) = 0\\
        & \Rightarrow L(-1)F(\omega) = 0
    \end{align*}
    Since $W$ is irreducible and not isomorphic to $V$, we know that from \cite{Li-vacuum-like} that $W$ contains no vacuum-like vectors. Therefore $\ker L(-1) = 0$ and $F(\omega) = 0$.
    \item If $h=0$, then $W = V$. The only nonzero weight 2 element in $V$ is simply scalars of $\omega$. So $F(\omega) = a\omega$ for some $a\in \C$. If $a\neq 0$, then $\omega\in O(V)$. The image of $\omega$ in $A(V)$ would then be zero. This contradicts the result that $A(V) = \C[x]/(G_{p,q}(x))$ in [W] where $x$ is the image of $\omega$ in $A(V)$. Thus $a=0$ and $F(\omega) = 0$. 
\end{enumerate}
\end{proof}

\begin{rema}
Here is a straightforward way to show the $h=0$ case: $F(\omega) = a\omega$ has to satisfy
$$F(\omega_1\omega) = \omega_1F(\omega) + F(\omega)_1\omega \Rightarrow a L(0) \omega = L(0) a \omega + a L(0)\omega \Rightarrow a L(0) \omega = 0 \Rightarrow a = 0$$
\end{rema}


\begin{rema}\label{strictly-less-2}
If $c\neq c_{pq}$ for any choices of $p,q$, then it is possible that $h=1$. Two interesting phenomena happen in this case. First, the zero-mode derivation given by the lowest weight vector is trivial, as we have
$$F(\omega) = w_0 \omega = -\omega_0 w + L(-1)\omega_1 w = -L(-1)w + L(-1)w = 0.$$
So in this case, $\dim Z^1(V, W) = 0 < W_{[1]} / L(-1)W_{[0]}$ (where $W_{[0]}=0$). Secondly, there exists a nontrivial derivation taking image in the (universal) $M(c,h)$, as will be shown in the following proposition. This provides an example where $H^1(V, W)\neq Z^1(V, W)$. 
\end{rema}

\begin{prop}\label{L(-1)w-der}
Let $c\in \C$ be a number that is not of the form $c_{pq}$ as in \ref{Formula-5}. Let $V= L(c, 0), W=M(c,1)$ and $w\in W$ be a nonzero lowest weight vector. Then the map 
$$F(\omega) = L(-1)w$$
extends to a well-defined derivation in $H^1(V, W)$. 
\end{prop}

\begin{proof}
Under our assumption on $c$, $L(c,0) = V(c, 0)$ (see \cite{LL} Section 6.1). By PBW theorem, the vectors 
$$\omega_{-n_1}\cdots \omega_{-n_s}\one = L(-n_1-1)\cdots L(-n_s-1)\one, s\in \N, n_1 \geq \cdots \geq n_s \geq 1$$
for a basis of $L(c,0)$. 

We extend the map $F$ to $V$ recursively by 
\begin{align*}
    F(\one) = 0, F(L(-n)\one) &= \frac 1 {(n-2)!}L(-1)^{n-1}w, n\geq 2\\
    F(L(-n)v) &= F(\omega)_{-n+1} v + \omega_{-n+1} F(v), v = L(-n_1)\cdots L(-n_s)\one, n \geq n_1. 
\end{align*}
Then $F$ is grading-preserving with $F(\omega) = F(L(-2)\one) = L(-1)w$. Moreover, $F(Y(\omega, x)\one) = e^{xL(-1)}F(\omega)$. To show that $F$ is a derivation, we first show that $$F(Y(\omega,x)v) = Y_W(\omega,x)F(v) + Y_{WV}^W(F(\omega), x)v$$
when $v=\omega_{-n_1}\cdots \omega_{-n_s}\one$ for some $n_1 \geq \cdots \geq n_s \geq 1$. We apply induction on $s$. For the base case $s=1$, it is clear that
\begin{align*}
F(\omega_n\omega) &= \omega_n F(\omega) + F(\omega)_n \omega
\end{align*}
for $n\geq 3$ and $n\leq 0$. When $n=2$ or $n=1$, note that $F(\omega_2\omega)=0$, $F(\omega_1\omega)=2F(\omega)$, while
\begin{align*}
    \omega_2 F(\omega) + F(\omega)_2 \omega &= \omega_2F(\omega)-\omega_2F(\omega) = 0\\
    \omega_1 F(\omega) + F(\omega)_1 \omega &= \omega_1 F(\omega) + \omega_1 F(\omega) - L(-1) \omega_2 F(\omega) = 2L(0)L(-1)w - L(-1)L(1)L(-1)w\\
    &= 2L(-1)w
\end{align*}
To perform the inductive step, it suffices to show that 
$$F(\omega_m v) = \omega_m F(v) + F(\omega)_m v$$
for $m\geq 0, v=\omega_{-n_1}\cdots \omega_{-n_s}\one$, $n_1\geq \cdots \geq n_s \geq 1$. Let $v^{(1)} = \omega_{-n_2}\cdots\omega_{-n_s}\one$. Then 
\begin{align}
    F(\omega_m v) &= \Res_{x_1} x_1^m\Res_{x_2} x_2^{-n_1} F(Y(\omega, x_1)Y(\omega, x_2)v^{(1)})\cdot \Res_{x_0} x_0^{-1} \delta\left(\frac{x_1-x_2}{x_0}\right) \nonumber\\
    &= \Res_{x_1} x_1^m\Res_{x_2} x_2^{-n_1} F(Y(\omega, x_2)Y(\omega, x_1)v^{(1)})\cdot \Res_{x_0} x_0^{-1} \delta\left(\frac{-x_2+x_1}{x_0}\right)\nonumber\\
    &\quad + \Res_{x_0} \Res_{x_1} x_1^m\Res_{x_2} x_2^{-n_1} F(Y(Y(\omega, x_0)\omega, x_2)v^{(1)})\cdot  x_1^{-1} \delta\left(\frac{x_2+x_0}{x_1}\right)\nonumber\\
    &= F(\omega_{-n_1}\omega_m v^{(1)}) \label{Compatibility-0-a}\\
    & \quad + \Res_{x_0}\Res_{x_2} (x_2+x_0)^m x_2^{-1} F(Y(Y(\omega,x_0)\omega,x_2)v^{(1)}). \label{Compatibility-0-b}
\end{align}
For (\ref{Compatibility-0-a}), since $m>0$, $\omega_m v^{(1)}$ is a linear combination of $\omega_{-p_1}\cdots \omega_{-p_r}\one$ for some $p_1\geq \cdots p_r \geq 1$ with $r\leq s-1$. Thus from the induction hypothesis, the first term
\begin{align}
    & \quad F(\omega_{-n_1}\omega_m v^{(1)}) = F(\omega)_{-n_1} \omega_m v^{(1)} + \omega_{-n_1} F(\omega)_m v^{(1)} + \omega_{-n_1}\omega_m F(v^{(1)}) \nonumber\\
    & = \Res_{x_2} x_2^{-n_1}\Res_{x_1}x_1^m \Res_{x_0} x_0^{-1}\delta\left(\frac{-x_2+x_1}{x_0}\right)\nonumber\\
    & \quad \cdot \left(Y_{WV}^W(F(\omega), x_2)Y(\omega, x_1)v^{(1)} + Y_W(\omega, x_2) Y_{WV}^W(F(\omega), x_1)v^{(1)} + Y_W(\omega, x_2) Y_W(\omega,x_1)F(v^{(1)})\right)\label{Compatibility-1}
\end{align}
For (\ref{Compatibility-0-b}), we rewrite it as 
$$\Res_{x_2} \sum_{i=0}^m \binom m i  x_2^{-n_1+m-i} F(Y(\omega_i \omega, x_2)v^{(1)})$$
where $\omega_i\omega$ has only three nonzero options: $\omega_0\omega = L(-1)\omega, \omega_1\omega = 2\omega, \omega_3\omega = c \one / 2$. Thus, (\ref{Compatibility-0-a}) can be rewritten as 
\begin{align}
\Res_{x_2} \sum_{i=0}^m \binom m i x_2^{-n_1+m-i} \left(Y_W(\omega_i\omega, x_2)F(v^{(1)}) + Y_{WV}^W (F(\omega)_i\omega, x_2) v^{(1)}+Y_{WV}^W (\omega_iF(\omega), x_2) v^{(1)}\right) \label{Compatibility-2}
\end{align}
using the induction hypothesis. Here the process for the $i = 0$ is slightly more complicated, using $L(-1)$-derivative property, integration by parts formula $\Res_x f(x)g'(x) = -\Res_x f'(x)g(x)$, and the property
$$F(\omega)_0 \omega = (L(-1)w)_0 \omega = 0$$
Details are shown here:
\begin{align*}
& \quad \Res_{x_2} x_2^m F(Y(\omega_0\omega, x_2)v^{(1)})= \Res_{x_2} x_2^m F(Y(L(-1)\omega, x_2)v^{(1)})\\
&= \Res_{x_2} x_2^m\frac{\partial}{\partial x_2} F(Y(\omega, x_2)v^{(1)})= -\Res_{x_2} \left(\frac{\partial}{\partial x_2} x_2^m\right) F(Y(\omega, x_2)v^{(1)}) \\
&= -\Res_{x_2} \left(\frac{\partial}{\partial x_2} x_2^m\right) \left(Y_{WV}^W(F(\omega), x_2) v^{(1)} + Y_W(\omega, x_2)F(v^{(1)})\right)\\
&= \Res_{x_2}  x_2^m \left(\frac{\partial}{\partial x_2}Y_{WV}^W(F(\omega), x_2) v^{(1)} +\frac{\partial}{\partial x_2} Y_W(\omega, x_2)F(v^{(1)})\right)\\
&= \Res_{x_2}  x_2^m \left(Y_{WV}^W(L(-1)F(\omega), x_2) v^{(1)} + 0 + Y_W(L(-1)\omega, x_2)F(v^{(1)})\right)\\
&=\Res_{x_2}  x_2^m \left(Y_{WV}^W(\omega_0 F(\omega), x_2) v^{(1)}  + Y_{WV}^W(F(\omega)_0\omega, x_2)v^{(1)} + Y_W(\omega_0 \omega, x_2)F(v^{(1)})\right)
\end{align*}
Now we rewrite (\ref{Compatibility-2}) as
\begin{align*}
    & \Res_{x_0}\Res_{x_2}x_2^{-n_1}(x_2+x_0)^m \cdot \Res_{x_1} x_1^{-1}\delta\left(\frac{x_2+x_0}{x_1}\right)\\ 
    & \cdot \left(Y_W(Y(\omega,x_0) x_2)F(v^{(1)}) + Y_{WV}^W (Y_{WV}^W(F(\omega), x_0)\omega, x_2) v^{(1)}+Y_{WV}^W (Y_W(\omega,x_0)F(\omega), x_2) v^{(1)}\right).
\end{align*}
Combined with the result (\ref{Compatibility-1}) we computed for (\ref{Compatibility-0-a}) and use Jacobi identity, we find that 
\begin{align*}
    F(\omega_m v) &= \Res_{x_2}x_2^{-n_1}\Res_{x_1}x_1^m \Res_{x_0} x_0^{-1} \delta \left(\frac{x_1-x_2}{x_0}\right)\\
    &\quad \cdot \left(Y_{WV}^W(F(\omega), x_1) Y(\omega, x_2) v^{(1)}+ Y_W(\omega, x_1)Y_{WV}^W (F(\omega),x_2) v^{(1)} + Y_W(\omega, x_1) Y_W(\omega, x_2) F(v^{(1)})\right)\\
    &= F(\omega)_m \omega_{-n_1} v^{(1)} + \omega_m F(\omega)_{-n_1} v^{(1)} + \omega_m \omega_{-n_1} F(v^{(1)}) \\
    &= F(\omega)_m v + \omega_m F(\omega_{-n_1}v^{(1)}) = F(\omega)_m v + \omega_m F(v).
\end{align*}
This finishes the proof of 
$$F(Y(u, x)v) = Y_W(u, x) F(v) + Y_{WV}^W (F(u), x) v$$
in case $u = \omega, v = \omega_{-n_1}\cdots \omega_{-n_s}\one$, $n_1 \geq \cdots \geq n_s \geq 1$.

The proof for the case when $u=\omega_{-p_1}\cdots \omega_{-p_r}\one$ is similar: the base case $r=1$ is checked with $L(-1)$-derivative property. Write $u^{(1)} = \omega_{-p_2}\cdots \omega_{-p_r}\one$ so that $u = \omega_{-p_1} u^{(1)}$. We then start from the iterate side of the Jacobi identity, write them as products and use the results on products to conclude the proof. We shall not include the details here. 
\end{proof}


\subsection{Irreducible module with negative lowest weight} 

The case when $h<0$ is much more interesting. Here we prove a partial result that $H^1(V, W) = Z^1(V, W)$ for $W = L(c,h)$ whenever $h\geq -3$. We first note the following:

\begin{lemma}\label{Image-L(-1)^4}
Let $F: V\to W$ be a derivation. Suppose that $F(\omega)\in \text{Im}L(-1)^4$, then $F(\omega)= 0$. So $F(v) = 0$ for any $v\in V$. 
\end{lemma}

\begin{proof}

We show that if $F(\omega)\in \text{Im}L(-1)^p$ for some $p\geq 4$, then $F(\omega)\in \text{Im}L(-1)^{p+1}$. Thus if $F(\omega)\in \text{Im}L(-1)^4$, then $F(\omega) \in \text{Im}L(-1)^p$ for every $p \geq 4$. This implies $F(\omega)=0$. 

It can be shown by a straightforward induction that for every $m\in \Z_+$ and every $p\in \Z_+$, 
$$L(m)L(-1)^p = \sum_{i=0}^p i!\cdot \binom p i  \binom{m+1} i  L(-1)^{p-i}L(m-i).$$
Now assume that $F(\omega)\in \text{Im}L(-1)^p$, i.e., $$F(\omega) = L(-1)^{p} w_{(-p+2)}$$
for some $w_{(-p+2)}$ of weight $-p+2$. From the assumption that $F$ is a derivation, 
$$F(\omega_p \omega) = \omega_p F(\omega) + F(\omega)_p\omega.$$
We compute the first term on the right-hand-side using the formula for $L(m)L(-1)^p$: 
\begin{align*}
    L(p-1)L(-1)^p w_{(-p+2)} & = \sum_{i=0}^p i! \binom p i \binom {p} i L(-1)^{p-i}L(p-1-i)w_{(-p+2)}\\
    & = p! L(-1)w_{(-p+2)} + p! \cdot p \cdot L(-1)\cdot L(0) w_{(-p+2)} \\
    & \qquad + \sum_{i=0}^{p-2} i! \binom p i \binom {p} i L(-1)^{p-i}L(p-1-i)w_{(-p+2)}\\
    &= p!(1 - p^2 + 2p) L(-1)w_{(-p+2)} \\
    & \qquad + L(-1)^2 \sum_{i=0}^{p-2} i! \binom p i \binom {p-2} i L(-1)^{p-i}L(p-1-i)w_{(-p+2)}
\end{align*}
From the general fact that 
$$(L(-1)w)_n = -n w_{n-1},$$
we know that
\begin{align*}
    F(\omega)_p \omega & = (L(-1)^p w_{(-p+2)})_p \omega = (-1)^p\cdot 
    p!\cdot (w_{(-p+2)})_0 \omega \\ 
    &= (-1)^p \cdot p! \left((-1)^{0+1}\omega_0  + (-1)^{1+1} L(-1)\omega_1  + L(-1)^2 \sum_{j=2}^\infty \frac{(-1)^{j+1}}{j!}L(-1)^{j-2}\omega_j\right)w_{(-p+2)}\\
    &= (-1)^p \cdot p! \left((-p+1) L(-1)  + L(-1)^2 \sum_{j=2}^\infty \frac{(-1)^{j+1}}{j!}L(-1)^{j-2}\omega_j\right)w_{(-p+2)}
\end{align*}
Thus from $F(\omega_p \omega)=0$ (note that $p\geq 4)$, we have
\begin{align*}
    0 &= \omega_p F(\omega) + F(\omega)_p \omega \\
    &= p!(1 - p^2 + 2p) L(-1)w_{(-p+2)} \\
    & \qquad + L(-1)^2 \sum_{i=0}^{p-2} i! \binom p i \binom {p-2} i L(-1)^{p-i}L(p-1-i)w_{(-p+2)}\\
    & \qquad +  (-1)^p \cdot p!(-p+1) L(-1) w_{(-p+2)} + (-1)^p \cdot p! L(-1)^2 \sum_{j=2}^\infty \frac{(-1)^{j+1}}{j!}L(-1)^{j-2}\omega_jw_{(-p+2)}
\end{align*}
Since $W$ is irreducible and not $L(0)$-isomorphic to $V$, we know from \cite{Li-vacuum-like} that $\ker L(-1)= 0$. Therefore, 
\begin{align}
    p!(1-p^2+2p + (-1)^p(1-p))w_{(-p+2)} &= L(-1)  \sum_{i=0}^{p-2} i! \binom p i \binom {p-2} i L(-1)^{p-i}L(p-1-i)w_{(-p+2)} \nonumber\\
    &\qquad + (-1)^p \cdot p! L(-1)^2 \sum_{j=2}^\infty \frac{(-1)^{j+1}}{j!}L(-1)^{j-2}\omega_jw_{(-p+2)}\label{Formula-2}
\end{align}
One see that the coefficient is nonzero whenever $p\geq 4$. Thus $w_{(-p+2)}\in \text{Im}L(-1)$, which implies that $F(\omega)\in \text{Im}L(-1)^{p+1}$.
\end{proof}

\begin{rema}
Note that the coefficient of $w_{(-p+2)}$ in (\ref{Formula-2}) is zero for $p=2$ and $p=3$. This is why we need $F(\omega)\in \text{Im} L(-1)^4$. 
\end{rema}

\begin{lemma}\label{L(2m)-action}
$F(\omega)\in \text{Im}L(-1)$. Moreover, for any $m\in \Z_+$,
\begin{align*}
    & L(2m)F(\omega)\in \text{Im}L(-1), \\ & L(2m)F(\omega)-\frac 1 2 L(-1)L(2m+1)F(\omega)\in \text{Im}L(-1)^2. 
\end{align*}
\end{lemma}

\begin{proof}
From $F\circ L(-1) = L(-1)\circ F$, we see that 
\begin{align*}
    0 &= F(\omega)_0 \omega = (-1)^{0+1}\omega_0 F(\omega) + L(-1)\omega_1 F(\omega) + L(-1)^2 \sum_{j=2}^\infty \frac{(-1)^{1+j}}{j!}L(-1)^{j-2} \omega_{j}F(\omega)\\
    &= L(-1)F(\omega)  + L(-1)^2\sum_{j=2}^\infty \frac{(-1)^{1+j}}{j!} L(-1)^{j-2}L(j-1)F(\omega).
\end{align*}
Then from $\ker L(-1)=0$, we see that 
$$ F(\omega) = L(-1)\sum_{j=2}^\infty \frac{(-1)^{1+j}}{j!} L(-1)^{j-2}L(j-1)F(\omega)\in \text{Im} L(-1). $$

In general, for $m\in \Z_+$
\begin{align*}
    0 &= \omega_{2m} F(\omega) + F(\omega)_{2m} \omega \\
    &= \omega_{2m} F(\omega) + (-1)^{2m+1} \omega_{2m} F(\omega) + L(-1)\omega_{2m+1} F(\omega) + \sum_{j=2}^\infty \frac{(-1)^{j+1+2m+1}}{j!}L(-1)^j \omega_{2m+j}F(\omega)\\
    &= L(-1)L(2m)F(\omega) -  \frac 1 {2!} L(-1)^2 L(2m+1)F(\omega) + L(-1)^3\sum_{j=3}^\infty \frac{(-1)^{j+3}}{j!}L(-1)^{j-3} \omega_{2m+j}F(\omega)
\end{align*}
From $\ker L(-1)=0$, we see that 
$$L(2m)F(\omega) -  \frac 1 {2!} L(-1) L(2m+1)F(\omega) + L(-1)^2\sum_{j=3}^\infty \frac{(-1)^{j+3}}{j!}L(-1)^{j-3} \omega_{2m+j}F(\omega)=0$$
The conclusion then follows. 
\end{proof}

\begin{prop}\label{Vir-neg-energy-prop}
Let $V=L(c,0)$ and $W=L(c,h)$ with $c=c_{p,q}$ and $h=h_{m,n}$ as in (\ref{Formula-5}) and (\ref{Formula-6}). Let $F: V \to W$ be a derivation. Then for $h=-1, -2$ and $-3$, 
\begin{enumerate}
    \item $F(\omega)\in \text{Im}L(-1)^2$. 
    \item There exists $w_{(1)}\in W_{[1]}$ such that $F(\omega)- (w_{(1)})_0\omega \in \text{Im}L(-1)^4$. 
\end{enumerate}
\end{prop}

\begin{proof}
The following fact will be convenient in computations. Let $m, n, p \in \Z_+$. If $m\neq kn$ for any $k\in\Z_+$, then 
\begin{align}
    L(m)L(-n)^p = \sum_{i=0}^p \binom p i \cdot \prod_{j=1}^i (m+(2-j)n)\cdot L(-n)^{p-i}L(m-in). 
\end{align}
If $m=kn$ for some $k\in \Z_+$, then 
\begin{align}
    L(kn)L(-n)^p &= \sum_{i=0}^p \binom p i \cdot \prod_{j=1}^i (k+(2-j)n)\cdot L(-n)^{p-i}L(kn-in) \nonumber \\
    & \quad+ \binom p k \frac{(k+1)!}{2}n^{k-1}\frac{n^3-n}{12}c L(-n)^{p-k}. 
\end{align}
These formulas can be easily proved by induction. 

To give a sketch, we first write $F(\omega)$ as a linear combination of vectors 
\begin{align}
    L(-1)^{r_1}L(-2)^{r_2}\cdots L(-n)^{r_n}w \label{Formula-3}
\end{align}
where $w\in W_{[-h]}$ is the lowest weight vector of $W$, and $r_1 + 2r_2 + \cdots + nr_n = h+2$. Note that $F(\omega)\in \text{Im}L(-1)$ implies that $r_1\geq 1$. Then determine the coefficients using Lemma \ref{L(2m)-action}. 

We should note here that the vectors listed in (\ref{Formula-3}) are linearly independent. Indeed, since $r_1\geq 1$ and $\ker L(-1)=0$, it suffices to show that $M(c,h)$ contains no singular vectors when the conformal weight is less or equal to 1. The main tool is Kac determinant formula, which tells that the Gram matrix formed by vectors in $M(c,h)$ of conformal weight 1 has its determinant proportional to 
$$\prod_{\substack{r,s\in \Z_+,\\ 1\leq rs \leq -h+1}} (h-h_{r,s})^{P(-h+1-rs)}$$
where $P(n)$ is the number of partitions of $n$ (see \cite{Kac-determinant}, \cite{Feigin-Fuchs} and \cite{IK}). A singular vector exists only when this determinant vanishes, which happens only when $r$ and $s$ are chosen such that
\begin{align}
    h = h_{r,s}, -h+1-rs\geq 0. \label{Formula-4}
\end{align}
Recall that our $h$ is chosen to be $h_{m,n}$. It follows from an elementary computation that $h=h_{r,s}$ if and only if $r=m,s=n$ or $r=kp-m,s=kq-n$ for some positive integer $k$. Without loss of generality, we can assume that $2m \leq p$ and $2n \leq q$, so that any $(r,s)$ satisfying $h=h_{r,s}$ would also satisfy $r\geq m, s \geq n$, and thus $rs\geq mn$. 

We now compute $mn+h-1$. 
\begin{align*}
    mn+h-1 = mn + \frac{(np-mq)^2-(p-q)^2}{4pq} - 1 = \frac{(np+mq)^2-(p+q)^2}{4pq}
\end{align*}
Since $m\geq 1$ and $n\geq 1$, $np+mq \geq p+q > 0$. Thus the numerator is nonnegative, and thus $mn \geq 1-h$. This is to say that any $(r,s)$ satisfying $h=h_{r,s}$ would also satisfy $rs \geq mn >-h+1$. So there is no $(r,s)$ satisfying (\ref{Formula-4}). Thus the determinant is nonvanishing. 

A similar argument also shows that the determinant of the Gram matrix is nonvanishing on any homogeneous subspaces of $M(c,h)$ of conformal weight less or equal to 1. So there does not exist any singular vectors in $M(c,h)$ of conformal weight less or equal to 1. This implies that vectors of the form (\ref{Formula-3}) with $r_1, ..., r_n\geq 0$, $r_1 + 2r_2 + \cdots + nr_n \leq 1$ are all linearly independent.  

Now we proceed with our computation: 

\begin{enumerate}
    \item \textbf{When $\boldsymbol{h=-1}$,} we set 
$$F(\omega) = a_{12}L(-1)L(-2)w + a_{111} L(-1)^3 w.$$
Then 
\begin{align*}
    L(2)F(\omega) &= a_{12}\left(5+ \frac 1 2 c\right)L(-1)w + a_{111}\left(-2\cdot 3\cdot 2\right)L(-1)w,\\
    L(3)F(\omega) &= a_{12}\left(-4+\frac 1 2 c\right)w + a_{111}\left(-4\cdot 3\cdot 2\right)w,
\end{align*}
It turns out that
$$L(2)F(\omega)-\frac 1 2 L(-1)L(3)F(\omega)= a_{12}\left(7+\frac 1 4 c \right)L(-1)w.$$
From Lemma \ref{L(2m)-action}, we know $L(2)F(\omega)-\frac 1 2 L(-1)L(3)F(\omega)\in \text{Im}L(-1)^2=\{0\}$. So $a_{12}=0$ if $7+\frac 1 4 c \neq 0$. This is guaranteed by our choice of $c$. In fact, $7+\frac 1 4 c = 0$ implies that 
$$p = \frac{41\pm \sqrt{1537}}{12} q. $$
Obviously no integers $p,q$ can satisfy this. Thus we have $a_{12}=0$. So
$$F(\omega) = a_{111}L(-1)^3 w \in \text{Im}L(-1)^2. $$

\noindent\textbf{When $\boldsymbol{h=-2}$,} we set 
$$F(\omega) = a_{13}L(-1)L(-3)w + a_{112} L(-1)^2L(-2)w + a_{1111} L(-1)^4w. $$
Then 
\begin{align*}
    L(2)F(\omega)&= a_{13}\left( 3\cdot 4 L(-2) + 5 L(-1)^2\right)w + a_{112}\left(\left(4+\frac 1 2 c\right)L(-1)^2\right)w \\
    &\quad + a_{1111}\left((-8)\cdot 3 \cdot 2 L(-1)^2\right)w.
\end{align*}
From $L(2)F(\omega)\in \text{Im}L(-1)$,  $a_{13}=0$. Thus, 
$$F(\omega) = a_{112}L(-1)^2L(-2)w + a_{1111}L(-1)^4 w \in \text{Im}L(-1)^2.$$

\noindent\textbf{When $\boldsymbol{h=-3}$,} we set 
\begin{align*}
    F(\omega) &= a_{14}L(-1)L(-4)w + a_{122} L(-1)L(-2)^2w + a_{113}L(-1)^2L(-3)w \\
    &\quad + a_{1112} L(-1)^3L(-2)w + a_{11111} L(-1)^5 w. 
\end{align*}
Then 
\begin{align*}
    L(2)F(\omega) &= a_{14}\left(3\cdot 5 L(-3) + 6L(-1)L(-2)\right)w \\
    &\quad + a_{122} \left(-9L(-3)+(2+c)L(-1)L(-2)\right)w \\
    &\quad + a_{113}\left(5L(-1)^3+2\cdot 3 \cdot 4 L(-1)L(-2)\right)w \\
    & \quad + a_{1112} \left(\left(15+\frac 1 2 c\right)L(-1)^3 - 3\cdot 2\cdot 2 L(-1)L(-2)\right)w \\
    & \quad + a_{11111} \left(-10\cdot 3\cdot 2\cdot 2 L(-1)^3\right) w, \\
    L(3)F(\omega) &= a_{14}\left(4\cdot 6 L(-2) + 7L(-1)^2\right)w \\
    &\quad + a_{122} \left(5\cdot 3L(-1)^2+4(-16+c)L(-2)\right)w\\
    &\quad + a_{113}\left((22+2c)L(-1)^2\right)w  \\
    &\quad + a_{1112} \left((-36+6c)L(-1)^2 - 4\cdot 3 \cdot 2 L(-2)\right)w \\
    &\quad + a_{11111} \left(-25\cdot 4\cdot 3 \cdot 2 L(-1)^2\right) w. 
\end{align*}
Then from $L(2)F(\omega)\in \text{Im}L(-1)$, we see that 
$$15 a_{14}-9a_{122} = 0.$$
From $L(2)F(\omega)-\frac 1 2 L(-1)L(3)F(\omega)\in \text{Im}L(-1)^2$, we see that 
$$-6a_{14}+(34-c)a_{122}=0 $$
(amazingly the variables $a_{113}, a_{1112}$ and $a_{11111}$ all have zero coefficients). The matrix of this system is degenerate only when $c=188/5$. So for our choice of $c$, the matrix is nondegenerate. Therefore, $a_{14} = a_{122} = 0$, and   
$$F(\omega) = a_{113}L(-1)^2L(-3)w  + a_{1112} L(-1)^3L(-2)w + a_{11111} L(-1)^5 w \in \text{Im}L(-1)^2. $$

\item We note first that such $w_{(1)}$ has to be outside of Im$L(-1)$, as the zero-mode of elements in Im$L(-1)$ is identically zero. 

\noindent\textbf{When $\boldsymbol{h=-1}$,} the only choice of elements in $W_{[1]}$ is (up to a scalar) $L(-2)w$. It can be computed that 
$$(L(-2)w)_0 \omega = \left(-\frac {13}  6 + \frac 1 {12} c\right)L(-1)^3 w. $$
The coefficient of $L(-1)^3 w$ is zero only when $c = 26$ that does not meet our choice. 
So $F(\omega) = a_{111} \left(-\frac {13}  6 + \frac 1 {12} c\right)^{-1} (L(-2)w)_0\omega. $

\noindent\textbf{When $\boldsymbol{h=-2}$,} the only choice of elements in $W_{[1]}$ is (up to a scalar) $L(-3)w$. It can be computed that 
$$(L(-3)w)_0 \omega = -2\left(L(-1)^2L(-2) + \frac 1 {12} L(-1)^4 \left(-8+\frac 1 2 c\right)\right)w.$$
Looking back at $F(
\omega)$, from Lemma \ref{L(2m)-action} we know that $L(4)F(\omega) \in \text{Im} L(-1) = \{0\}$. It can be computed that 
$$L(4)F(\omega) = a_{112}\left(-8+\frac 1 2 c\right) + a_{1111}(-12)=0 \Rightarrow a_{1111} = \frac 1 {12} \left(-8+\frac 1 2 c\right)a_{112}.$$
So that 
$$F(\omega) = a_{112} \left(L(-1)^2 L(-2) + \frac 1 {12} \left(-8+\frac 1 2 c\right)L(-1)^4\right)w = -\frac 1 2 a_{112} (L(-3)w)_0 w. $$
\noindent\textbf{When $\boldsymbol{h=-3}$,} the only choice of elements in $W_{[1]}$ is some linear combination of $L(-4)w$ and $L(-2)^2w$. It can be computed that 
\begin{align*}
    (L(-4)w)_0\omega &= -\frac 5 2 L(-1)^2 L(-3)w + L(-1)^3 L(-2)w \\
    &\quad + \frac 1 {120}(-59+5c)L(-1)^5 w\\
    (L(-2)^2w)_0\omega &= \frac 3 2 L(-1)^2 L(-3)w + \frac 1 6 (-50+c) L(-1)^3L(-2)w \\
    &\quad + \frac 1 {40} (-49+4c)L(-1)^5w. 
\end{align*}
Recall that in (1) we have already obtained that 
$$F(\omega) = a_{113}L(-1)^2L(-3)w + a_{1112}L(-1)^3L(-2)w + a_{11111}L(-1)^5 w. $$
Let $x_1, x_2$ be numbers such that
\begin{align*}
    a_{113} &= -\frac 5 2 x_1 + \frac 3 2 x_2\\
    a_{1112} &= x_1 + \frac 1 6 (-50+c) x_2.
\end{align*}
The linear system is degenerate when 
$$-\frac 5 2 \cdot 1 6 (-50+c) - \frac 3 2 = 0\Rightarrow c > 50$$
Thus our choice of $c$ makes a nondegenerate system. So such $x_1, x_2$ exists uniquely for any fixed $a_{113}$ and $a_{1112}$. Then we know that 
$$F(\omega) + (-x_1 L(-4)w - x_2 L(-2)^2 w)_0 \omega \in \text{Im}L(-1)^4. $$
\end{enumerate}

\end{proof}

Consequently, we conclude the following theorem. 

\begin{thm}\label{Vir-neg-energy-thm}
For $V=L(c,0)$, $c=c_{p,q}$ and $W=L(c,h)$ with $h = h_{m,n}\in \Z, h \geq -3$, $H^1(V, W) = Z^1(V, W)$. 
\end{thm}

\begin{proof}
Let $F: V\to W$ be any derivation. From Proposition \ref{Vir-neg-energy-prop}, we know that for some $w_{(1)}\in W_{[1]}$, $F(\omega)-(w_{(1)})_0\omega\in \text{Im}L(-1)^4$. Consider the map $G: V\to W$ defined by
\begin{align*}
    G(v) = F(v) - (w_{(1)})_0 v.
\end{align*}
Since both $v\mapsto F(v)$ and $v\mapsto (w_{(1)})_0 v$ is a derivation, we know that $G$ is also a derivation, satisfying $G(\omega)\in \text{Im}L(-1)^4$. It then follows from Lemma \ref{Image-L(-1)^4} that $G = 0$. Thus $F(v)= (w_{(1)})_0 v$
\end{proof}


\subsection{The module $L(c,h)$ with $\N$-grading}

In this subsection we consider $W=L(c,h)$ with $\N$-grading. One should note that $H^1(V, W)$ and $Z^1(V, W)$ are different from those with $L(0)$-gradings. 

\begin{thm}\label{Vir-canonical-N-grading}
Let $V = L(c, 0)$ with $c=c_{p,q}$ as in (\ref{Formula-5}). Let $W = L(c, h)$ with $h=h_{m,n}$ as in (\ref{Formula-6}) with the canonical $\N$-grading given by the operator $\d_W = L(0)-h$. Then $H^1(V, W) = 0$. \end{thm}

\begin{proof}
It suffices to assume that $h\neq 0$. Let $w$ be the lowest weight vector of $W$. Assume that $L(-2)w$ and $L(-1)^2 w$ are linearly independent. Set
$$F(\omega)=aL(-2)w + bL(-1)^2w. $$
Then 
\begin{align*}
    L(1)F(\omega) &= (3a+(4h+2)b)L(-1)w\\
    L(2)F(\omega) &= \left(\left(4h+\frac 1 2 c\right)a + 6hb\right)w.
\end{align*}
From 
$$F(L(0)\omega) = L(0)F(\omega) + F(\omega)_1 \omega = 2L(0)F(\omega) - L(-1)L(1)F(\omega) + \frac 1 2 L(-1)^2L(2)F(\omega), $$
we have
\begin{align*}
    & 2F(\omega) = 2(h+2)F(\omega)-L(-1)L(1)F(\omega)+\frac 1 2 L(-1)^2L(2)F(\omega)\\
    \Rightarrow & 0 = 2(h+1)a L(-2)w + \left(2(h+1)b - 3a - (4h+2)b + \frac 1 2 \left(4h+\frac 1 2c\right)a + 3hb\right)L(-1)^2w
\end{align*}
If $h\neq -1$, then $a=0$. What remains simplifies to $hb=0$. Since $h\neq 0$, $b=0$. On the other hand, if $h=-1$. In this case, the relation simplifies to 
$$b=\left(-5+\frac 1 4 c\right)a.$$
On the other hand, from 
$$F(L(1)\omega) = L(1)\omega + F(\omega)_2\omega = L(-1) L(2)F(\omega)=0, $$
and the fact that $\ker L(-1) = 0$, we see that 
$$\left(-4+\frac 1 2 c\right)a - 6b = 0$$
Together with the previous relation on $a,b$, we conclude that $(26-c)a = 0$. Since $c\neq 26$, we conclude that $a=0$ and thus $b=0$. 

It remains to study the case when $L(-2)w$ and $L(-1)^2w$ are linearly independent. Since the coefficient of $L(-1)^2w$ in the relation is nonzero (see \cite{Ash}), it suffices to consider the case when 
$$F(\omega) = aL(-2)w.$$
We similarly argue that $a=0$ if $h\neq -1$, and in case $h=-1$, we have $(-4+c/2)a=0$. Since $c\neq 8$, we still get $a=0$. 
\end{proof}

\begin{thm}
Let $V = L(c, 0)$ with $c=c_{p,q}$ as in (\ref{Formula-5}). Let $W = L(c, h)$ with $h=h_{m,n}$ as in (\ref{Formula-6}) with arbitrary $\N$-grading. Then $H^1(V, W) = Z^1(V, W)$.
\end{thm}

\begin{proof}
Let $\Xi$ be the lowest weight of $W$. The $\Xi=0$ case is proved in Theorem \ref{Vir-canonical-N-grading}. 
If $\Xi = 1$, then $F(\omega)=L(-1)w$ for some $w\in W_{[1]}$, which is zero automatically if $W$ is $L(0)$-isomorphic to $V$. In case $W$ is not $L(0)$-isomorphic to $V$, we note that 
$$w_0 \omega = -\omega_0 w + L(-1)\omega_1 w = (h-1)L(-1)w.$$
Since $h\neq 1$, we see that $F$ coincides with the zero-mode derivation defined by $w/(h-1)$. Thus $F$ is a zero-mode derivation. 

If $\Xi = 2$, then $F(\omega) = w$ for some $w\in W_{[2]}$. But then, 
$$2w = F(\omega_1\omega) = \omega_1 F(\omega) + F(\omega)_1 \omega = 2 \omega_1 F(\omega)= 2h w.$$
Since $h\neq 1$, we see that $w=0$. Thus $F = 0$. 
\end{proof}

\section{First cohomologies of lattice VOAs}\label{Section-4}

In this section, we will study the first cohomologies of the lattice VOA associated with an positive definite even lattice. We will compute all derivations using Proposition \ref{Der-gen} and Remark \ref{Rmk-2-12}. More precisely, we first assume that $F$ sends a generator to a linear combination of homogeneous elements of the same weight, then determine the coefficients and show that they coincide with some zero-mode derivations. 

\subsection{Lattice VOA and modules} The lattice VOA was first constructed in \cite{FLM}. The modules were classified in \cite{Dong-lattice}. We briefly review the construction of lattice VOA and modules following \cite{LL}. 

Let $L_0$ be an even lattice of rank $r$, i.e., $L_0$ is a free abelian group of rank $r$ equipped with a positive definite symmetric $\mathbb{Z}$-bilinear form satisfying 
$$\langle \alpha, \alpha\rangle \in 2\Z, \alpha\in L_0.$$
Let $\h = L_0\otimes_\Z \C$ and $L$ be the dual lattice of $L_0$, i.e., 
$$L = \{\alpha\in \h: \langle \alpha, \beta\rangle \in \Z, \forall \beta\in L_0\}. $$
Let $\hat{L}$ be a central extension of $L$ by a finite cyclic group $\langle \kappa \rangle$ of order $s$, i.e., $\hat{L}$ can be fitted into the following short exact sequence of groups:
$$1 \to \langle \kappa \rangle \to \hat{L} \xrightarrow{-} L \to  1, $$
where the map $\hat{L}\to L$ is denoted by $a\mapsto \bar{a}$. Let $e: L \to \hat{L}, \alpha\mapsto e_\alpha$ be a section of the short exact sequence i.e.,  $$ e_0 = 1; \overline{e_\alpha} = \alpha, \alpha\in L.$$
Let $\epsilon_0: L \times L \to \Z/s\Z$ be the corresponding 2-cocycle assoicated with the section $e$, i.e., for $\alpha, \beta\in L$, 
$$e_\alpha e_\beta = \kappa^{\epsilon_0(\alpha,\beta)}e_{\alpha+\beta}. $$
Fix some primitive $s$-th root of unity $\omega_s\in \C$, let $\C_{\omega_s}$ be the one-dimensional vector space on which $\kappa$ acts by the scalar $\omega_s$. Let 
$$\C\{L\} = \C[\hat{L}]\otimes_{\C[\kappa]}\C_{\omega_s}$$
be the induced $\hat{L}$-module. Set $\iota: \hat{L}\to \C\{L\}$ by $\iota(a) = a\otimes 1$. It is known that $\iota$ is an injection, and $\iota(\kappa b) = \omega_s \iota(b)$ is the only linear relations among $\iota(b), b\in \hat{L}$. Thus $\C[L]$ and $\C\{L\}$ are linearly isomorphic and can be (linearly) identified via $\alpha\mapsto \iota(e_\alpha)$. Define $\epsilon: L\times L \to \C^\times$ by 
$$\epsilon(\alpha, \beta) = \omega_s^{\epsilon_0(\alpha, \beta)}e_{\alpha+\beta}, \alpha, \beta\in L.$$
Then the action of $\hat{L}$ on $\C\{L\}$ can be described by
$$e_\alpha\cdot \iota(e_\beta) = \epsilon(\alpha, \beta)\iota(e_{\alpha+\beta}), \kappa \iota(e_\beta) = \omega_s \iota(e_\beta), \alpha, \beta\in L $$

Let $M(1) = S(\hat{\h}_-)$ be the Heisenberg VOA assoicated with $\h$ with level 1. In more detail, we view $\h$ as the abelian Lie algebra and consider its affinization 
$$\hat{h} = \h\otimes \C[t, t^{-1}] \oplus \C k$$
Let $\C \one_\h$ be the one-dimensional vector space on which $h(n)$ acts trivially for $n\geq 0$, and $k$ acts by the scalar 1. $M(1)$ is nothing but the induced module 
$$M(1) = U(\hat{\h})\otimes_{U(\h\otimes \C[t] \oplus \C k)} \C \one_\h $$
together with an appropriately defined vertex operator. To construct the lattice VOA, we consider the vector space
$$V_L = M(1) \otimes \C\{L\}$$
and the vacuum element 
$$\one = \one_\h \otimes 1,$$
together with the following actions
\begin{align*}
    a\in \hat{L}:& v\otimes \iota(e_\alpha) \mapsto v\otimes \iota(a\cdot e_\alpha),\\
    k \in \hat{\h}: & v\otimes \iota(e_\alpha)\mapsto v\otimes \iota(e_\alpha),\\
    h\otimes t^0 \in \hat{\h}: & v\otimes \iota(e_\alpha)\mapsto \langle h, \alpha\rangle v\otimes \iota(e_\alpha), \\
    h\otimes t^n \in \hat{\h}: & v\otimes \iota(e_\alpha)\mapsto h(n)v \otimes \iota(e_\alpha), (n\neq 0), 
\end{align*}
where $h\in \h, v\in M(1). $ So 
$$V_L = span\{a^{(1)}(-n_1)\cdots a^{(m)}(-n_m) \iota(e_\alpha): a^{(1)}, ..., a^{(m)}\in \h, n_1, ..., n_m \in \Z_+ \}$$
For $h\in \h$ and a formal variable $x$, we also define the action of 
$$x^h: v\otimes \iota(e_\alpha)\mapsto x^{\langle h, \alpha\rangle} v\otimes \iota(e_\alpha). $$
Using these actions, we define the following fundamental vertex operator associated with $\alpha\in L$ ($\subset \h$):
$$Y(\iota(e_\alpha), x) = \exp\left(-\sum_{n<0}\frac{\alpha(n)}{n} x^{-n}\right)\exp\left(-\sum_{n>0}\frac{\alpha(n)}{n}x^{-n}\right)e_\alpha x^{\alpha}.$$
The vertex operator for general elements is given by 
\begin{align*}
    & Y(a^{(1)}(-n_1)\cdots a^{(m)}(-n_m) \iota(e_\alpha), x)\\
    & \quad = \nord \frac 1{(n_1-1)!} \frac{d^{n_1-1}}{dx^{n_1-1}}a^{(1)}(x) \cdots \frac 1{(n_m-1)!} \frac{d^{n_m-1}}{dx^{n_m-1}}a^{(m)}(x) Y(\iota(e_\alpha), x)\nord 
\end{align*}
For any subset $E\subset L$, let 
$$\hat{E} = \{a\in \hat{L}: \bar a \in E\}, \C\{E\} = span\{\iota(a): \bar a \in E\}$$
and 
$$V_{E} = M(1) \otimes \C\{E\}\subset V_L.$$
Particularly interesting cases includes $E=L_0$ and $E=\gamma + L_0$ for some $\gamma\in L$, giving correspondingly $V_{L_0}$ and $V_{\gamma+L_0}$. 

It follows from \cite{LL} Section 6.4, 6.5, and \cite{DLM-Regular} (see also \cite{Dong-lattice}) that
\begin{enumerate}
    \item $V_{L_0}$ forms a VOA generated by
    $$a(-1)\one, a\in \h; \iota(e_\alpha), \alpha\in L_0. $$
    For $a^{(1)}, ... a^{(r)}\in \h, \alpha\in L$, 
    $$\wt\left(a^{(1)}(-n_1)\cdots a^{(m)}(-n_m) \iota(e_\alpha)\right) = n_1 + \cdots + n_m + \frac{\langle \alpha, \alpha\rangle}{2}.$$
    \item Up to isomorphisms of VOAs, $V_{L_0}$ is independent of the choice of $s\in \Z_+$, of the choice of $\omega_s$, and of the central extension of $L$. We can pick $s=2$, $\omega_s = -1$ and the central extension with a section satisfying $$\epsilon(\alpha, \beta) / \epsilon (\beta, \alpha) = (-1)^{\langle \alpha, \beta\rangle}. $$
    \item $V_L$ forms a $V_{L_0}$-module. Any irreducible $V_{L_0}$-module is isomorphic to $V_{\gamma+L_0}\subset V_L$ for some $\gamma\in L$. Let $\gamma_1, ..., \gamma_s\in L$ be a choice of representatives of $L/L_0$. Then the modules $V_{\gamma_1+L_0}, ..., V_{\gamma_s+L_0}$ form the set of equivalent classes of irreducible $V_{L_0}$-modules. 
    \item Every weak $V_{L_0}$-module is a direct sum of irreducible $V_{L_0}$-modules. 
\end{enumerate}

\subsection{The algebra $V_{L_0}$} 

\begin{lemma}
Let $\alpha_1, ..., \alpha_r \in L_0$ be a $\Z$-basis for $L_0$. Then $V_{L_0}$ is generated by $\alpha_i(-1)\one$ and $\iota(e_{\alpha_i})$, $i=1, ..., r$. 
\end{lemma}

\begin{proof}
Note that for any $\alpha, \beta\in L_0$, 
\begin{align*}
    Y(\iota(e_\alpha), x) \iota(e_\beta) &= \exp\left(-\sum_{n<0}\frac{\alpha(n)}{n}x^{-n}\right)e_\alpha x^{\alpha} \cdot \iota(e_\beta)\\
    &= \epsilon(\alpha,\beta)x^{\langle \alpha, \beta\rangle}\exp\left(-\sum_{n<0}\frac{\alpha(n)}{n}x^{-n}\right)\iota(e_{\alpha+\beta})
\end{align*}
The coefficient of the lowest power of $x$ is precisely $\epsilon(\alpha, \beta)\iota(e_{\alpha+\beta})$, where  $\epsilon(\alpha, \beta) \neq 0$. Thus $\iota(e_{\alpha+\beta})$ is generated by $\iota(e_\alpha)$ and $\iota(e_\beta)$. Therefore, $\iota(e_{\alpha_1}), ..., \iota(e_{\alpha_r})$, together with $h(-1)\one, h\in \h$, generates $V_{L_0}$. The conclusion follows by noticing that $h(-1)\one$ is a linear combinations of $\alpha_1(-1)\one, ..., \alpha_r(-1)\one$. 
\end{proof}

\begin{thm}\label{lattice-L(0)-thm}
Let $V_{L_0}$ be the lattice VOA associated with a positive definite even lattice. Let $F: V_{L_0}\to V_{L_0}$ be a derivation. Then $F$ is a zero-mode derivation. Consequently, $H^1(V_{L_0}, V_{L_0}) = Z^1(V_{L_0}, V_{L_0})$. 
\end{thm}

\begin{proof}
Fix any $h\in \h$. Since $F$ commutes with $L(0)$, $F(h(-1)\one)$ is of conformal weight 1. Thus we write
$$F(h(-1)\one) = \sum_{i=1}^r x_i(h) \alpha_i(-1)\one + \sum_{\alpha\in L_0, \langle \alpha,\alpha\rangle = 2} y_\alpha(h) \iota(e_\alpha). $$
where $x_i, y_\alpha\in \h^*$ for each $i=1, ..., r$ and $\alpha\in L_0, \langle \alpha,\alpha\rangle = 2$. 
Since $F$ is a derivation, for any $h, h_1\in \h$, we have
$$F(h_1(0)h(-1)\one ) = h_1(0)F(h(-1)\one) + F(h_1(-1)\one)_0 h(-1)\one $$
This is to say
$$0 = \sum_{\alpha\in L_0, \langle \alpha,\alpha\rangle = 2} \left(y_\alpha(h) \langle h_1, \alpha\rangle - y_\alpha(h_1) \langle h, \alpha\rangle \right) \iota(e_\alpha)$$
Thus for each index $\alpha$, 
$$y_\alpha(h)\langle h_1, \alpha\rangle = y_\alpha (h_1)\langle h, \alpha\rangle. $$
Since there are only finitely many $\alpha$ in the sum, there exists $h_1\in \h$ such that $\langle h_1, \alpha\rangle \neq 0$ for every $\alpha\in L_0, \langle \alpha, \alpha\rangle = 2$. Let 
$$t_\alpha = \frac{y_\alpha(h_1)}{\langle h_1, \alpha\rangle}. $$
Then for every $h\in h$ and every $\alpha\in L_0, \langle \alpha, \alpha \rangle = 2$, $$y_\alpha(h) = \langle h, \alpha\rangle t_\alpha. $$
Let $F_1: V\to V$ be the linear map defined by
$$F_1(v) = F(v) + \sum_{\alpha\in L_0, \langle \alpha,\alpha\rangle = 2} t_\alpha \iota(e_\alpha)_0 v. $$
$F_1$ is also a derivation since for each $\alpha\in L_0, \langle \alpha, \alpha\rangle = 2$, the map $v\mapsto \iota(e_\alpha)_0 v$ is the zero-mode derivation assoicated with the weight-1 element $\iota(e_\alpha)$. Then 
\begin{align*}
F_1(h(-1)\one) &= F(h(-1)\one) + \sum_{\alpha\in L_0, \langle \alpha,\alpha\rangle = 2} t_\alpha \iota(e_\alpha)_0 h(-1)\one \\
&= \sum_{i=1}^r x_i(h)\alpha_i(-1)\one + \sum_{\alpha\in L_0, \langle\alpha,\alpha\rangle = 2}\left( y_\alpha(h) - t_\alpha \langle h, \alpha\rangle \right)\iota(e_\alpha)\\
&= \sum_{i=1}^r x_i(h)\alpha_i(-1)\one. 
\end{align*}
For convenience, we set $\tilde{h}=\sum_{i=1}^r x_i(h)\alpha_i \in \h$. So that $F_1(h(-1)\one) = \tilde{h}(-1). $

Since $F_1$ commutes with $L(0)$, for every $i = 1, ..., r$, $F_1(\iota(e_{\alpha_i}))$ is of weight $\langle \alpha_i, \alpha_i\rangle/2$. Thus
\begin{align*}
    & F_1(\iota(e_{\alpha_i})) \\
    & = \sum_{\substack{\alpha\in L_0, \\ \langle \alpha,\alpha\rangle = \langle \alpha_i, \alpha_i\rangle}} z_\alpha \iota(e_\alpha) \\
    & +   \sum_{\substack{\beta\in L_0, s_1, ..., s_r \in \N,\\ k_1^{(1)}, ..., k_{s_1}^{(1)}, ..., k_1^{(r)}, ..., k_{s_r}^{(r)}\in \N\\ \sum_{p=1}^r \sum_{j=1}^{s_i} j\cdot k_j^{(p)} +  \frac 1 2 \langle \beta, \beta\rangle = \frac 1 2 \langle \alpha_i, \alpha_i\rangle}}& & u_{k_1^{(1)}, ..., k_{s_1}^{(1)}, ..., k_1^{(r)}, ..., k_{s_r}^{(r)}, \beta} \\
    & &\cdot & \alpha_1(-1)^{k_1^{(1)}}\cdots \alpha_1(-s_1)^{k_{s_1}^{(1)}}\cdots \alpha_r(-1)^{k_1^{(r)}}\cdots \alpha_r(-s_r)^{k_{s_r}^{(r)}}\iota(e_\beta). 
\end{align*}
for some $z_\alpha\in \C$ and $u_{k_1^{(1)}, ..., k_{s_1}^{(1)}, ..., k_1^{(r)}, ..., k_{s_r}^{(r)}, \beta}\in \C$.

Since $F_1$ is a derivation, for any $h\in \h$, we have
$$F(h(0) \iota(e_{\alpha_i})) = h(0) F(\iota(e_{\alpha_i})) + F(h(-1)\one)_0 \iota(e_{\alpha_i})$$
This is to say that
\begin{align*}
    & \sum_{\substack{\alpha\in L_0, \\ \langle \alpha,\alpha\rangle = \langle \alpha_i, \alpha_i\rangle}}  \langle h, \alpha_i\rangle \cdot z_\alpha \iota(e_\alpha) \\
    & +  \sum_{\substack{\beta\in L_0, s_1, ..., s_r \in \N,\\ k_1^{(1)}, ..., k_{s_1}^{(1)}, ..., k_1^{(r)}, ..., k_{s_r}^{(r)}\in \N\\ \sum_{p=1}^r \sum_{j=1}^{s_i} j\cdot k_j^{(p)} +  \frac 1 2 \langle \beta, \beta\rangle = \frac 1 2 \langle \alpha_i, \alpha_i\rangle}}& & \langle h, \alpha_i\rangle \cdot u_{k_1^{(1)}, ..., k_{s_1}^{(1)}, ..., k_1^{(r)}, ..., k_{s_r}^{(r)}, \beta} \\
    & &\cdot & \alpha_1(-1)^{k_1^{(1)}}\cdots \alpha_1(-s_1)^{k_{s_1}^{(1)}}\cdots \alpha_r(-1)^{k_1^{(r)}}\cdots \alpha_r(-s_r)^{k_{s_r}^{(r)}}\iota(e_\beta)\\ 
    = &  \sum_{\substack{\alpha\in L_0, \\ \langle \alpha,\alpha\rangle = \langle \alpha_i, \alpha_i\rangle}}\langle h, \alpha\rangle\cdot  z_\alpha \iota(e_\alpha) \\
    & +   \sum_{\substack{\beta\in L_0, s_1, ..., s_r \in \N,\\ k_1^{(1)}, ..., k_{s_1}^{(1)}, ..., k_1^{(r)}, ..., k_{s_r}^{(r)}\in \N\\ \sum_{p=1}^r \sum_{j=1}^{s_i} j\cdot k_j^{(p)} +  \frac 1 2 \langle \beta, \beta\rangle = \frac 1 2 \langle \alpha_i, \alpha_i\rangle}}& & \langle h, \beta \rangle\cdot  u_{k_1^{(1)}, ..., k_{s_1}^{(1)}, ..., k_1^{(r)}, ..., k_{s_r}^{(r)}, \beta} \\
    & &\cdot & \alpha_1(-1)^{k_1^{(1)}}\cdots \alpha_1(-s_1)^{k_{s_1}^{(1)}}\cdots \alpha_r(-1)^{k_1^{(r)}}\cdots \alpha_r(-s_r)^{k_{s_r}^{(r)}}\iota(e_\beta) \\ 
    & + \langle \tilde{h}, \alpha_i\rangle \iota(e_{\alpha_i})
\end{align*} 
Comparing the coefficient of $\iota(e_{\alpha_i})$ on both sides, we see that $\langle \tilde{h}, \alpha_i\rangle = 0$ for every $h\in \h$ and every $i = 1,..., r$. This implies that $\tilde{h}=0$ for every $h\in \h$. Thus $F_1(h(-1)\one) = 0$. And only the first four lines remain. 
Comparing the coefficients of other terms, we see that 
\begin{enumerate}
    \item For every $\alpha\neq \alpha_i, \langle \alpha, \alpha\rangle = \langle \alpha_i, \alpha_i \rangle, $
    $$\langle h, \alpha_i\rangle z_\alpha = \langle h, \alpha\rangle z_\alpha$$
    holds for every $h\in \h$. This is only possible when $z_\alpha = 0. $
    \item For every
    $\beta\in L_0, \langle \beta, \beta\rangle < \langle \alpha, \alpha\rangle$ and every possible choice of indices $k_1^{(1)}, ..., k_{s_1}^{(1)},$ $ ..., k_1^{(r)}, ..., k_{s_r}^{(r)}\in \N$, 
    $$\langle h, \alpha_i\rangle u_{k_1^{(1)}, ..., k_{s_1}^{(1)}, ..., k_1^{(r)}, ..., k_{s_r}^{(r)}, \beta} = \langle h, \beta\rangle u_{k_1^{(1)}, ..., k_{s_1}^{(1)}, ..., k_1^{(r)}, ..., k_{s_r}^{(r)}, \beta}$$
    holds for any $h\in \h$. This is possible only when $u_{k_1^{(1)}, ..., k_{s_1}^{(1)}, ..., k_1^{(r)}, ..., k_{s_r}^{(r)}, \beta} = 0$. 

\end{enumerate}

Thus we managed to show that 
$$F_1(h(-1)\one) = 0, F_1(\iota(e_{\alpha_i}))=z_{\alpha_i}\iota(e_{\alpha_i}), $$
for $h\in \h, i = 1, ..., r.$

Let $F_2: V \to V$ be the linear map defined by 
$$F_2(v) = F_1(v) - \sum_{i=1}^r z_{\alpha_i}(\alpha_i^\vee(-1)\one)_0 v$$
$F_2$ is a derivation since for every $i = 1, ..., r$, the map $v\mapsto (\alpha_i(-1)^\vee)_0 v$ is the zero-mode derivation assoicated with the weight-1 element $\alpha_i^\vee(-1)\one$. Then 
$$F_2(h(-1)\one) = 0, F_2(\iota(e_{\alpha_i})) = 0.$$
So $F_2$ is a derivation sending every generator of $V_{L_0}$ to zero. Thus $F_2 = 0$. In other words, 
$$0 = F_1(v) - \sum_{i=1}^r z_{\alpha_i}(\alpha_i^\vee(-1)\one)_0 v = F(v) + \sum_{\alpha\in L_0, \langle \alpha,\alpha\rangle = 2} x_\alpha\iota(e_\alpha)_0 v - \sum_{i=1}^r z_{\alpha_i}(\alpha_i^\vee(-1)\one)_0 v.$$
So $F$ is a zero-mode derivation. 
\end{proof}

\begin{rema}
For the lattice VOA $V_{L_0}$ assoicated with a positive definite even lattice, the associated Zhu's algebra $A_0(V_{L_0})$ is studied in \cite{DLM-Zhu-lattice}. If the dual lattice $L$ of $L_0$ contains a $\Z$-basis consisting of lattice points outside of $L_0$, then we can show that $O(V_L)$ contains no homogeneous elements of conformal weight 1. In this case, using the conclusions in Proposition \ref{Prop-2-23} and Corollary \ref{Corollary-1-22}, we can show that $w$ must be a linear combination of weight-1 elements. The computation is much less technical than here. However, the assumption generally does not hold: if $L_0$ is a unimodular lattice (e.g. $E_8$ or Leech lattice), from the representation theory of $A(V_{L_0})$ discussed in \cite{DLM-Zhu-lattice}, one sees that $A_0(V_L)=\C\one$ with $O_0(V_{L_0})$ containing all elements except the vacuum $\one$. The structure of $A_0(V_{L_0})$ is too trivial to give any useful information (cf. Remark \ref{Rmk-1-23}).  
\end{rema}

\subsection{The module $V_{\gamma+L_0}$ with $L(0)$-grading}

\begin{thm}
Fix an arbitrary $\gamma\in L, \gamma\notin L_0$. Then derivation $F: V_{L_0}\to V_{\gamma+L_0}$ is a zero-mode derivation. Consequently, $H^1(V_{L_0}, V_{\gamma+L_0}) = Z^1(V_{L_0}, V_{\gamma+L_0})$. 
\end{thm}

\begin{proof}
The computation is very similar. Instead of giving out every detail, we will only give a sketch. 

\begin{enumerate}
    \item For any $h\in \h$, let 
    $$F(h(-1)\one) = \sum_{\alpha\in L_0, \langle\gamma+\alpha, \gamma+\alpha\rangle = 2} x_\alpha(h) \iota(e_{\gamma+\alpha}). $$
    for some $x_\alpha\in \h^*$. 
    From $F(h_1(0)h(-1)\one)=h_1(0)F(h(-1)\one) + F(h_1(-1)\one)_0 h(-1)$, we similarly see that for every $\alpha\in L$ with $\langle\gamma+\alpha, \gamma+\alpha\rangle = 2$, 
    $$x_\alpha(h)\langle h_1, \alpha+\gamma\rangle = x_\alpha(h_1)\langle h,  \gamma+\alpha\rangle. $$
    Picking $t_\alpha=x_{\alpha}(h_1)/\langle h_1, \gamma+\alpha\rangle$ for some $h_1\in \h$ that makes the denominator nonzero for every $\alpha$ involved, and let 
    $$F_1(v) = F(v) - \sum_{\alpha\in L_0, \langle \gamma+\alpha, \gamma+\alpha\rangle = 2} t_\alpha \iota(e_{\gamma+\alpha})_0 v.$$
    Then $F_1$ is also a derivation, and $F_1(h(-1)\one) = 0$ for any $h\in \h$. 
    \item For any $i = 1, ..., r$, let 
    \begin{align*}
        F_1(\iota(e_{\alpha_i})) &= \sum_{\beta\in L_0,  \langle \gamma+\beta, \gamma+\beta\rangle = \langle \alpha_i, \alpha_i\rangle} y_\beta \iota(e_{\gamma+\beta}) \\
        & \quad + \sum_{\substack{\beta\in L_0,  h_1, ..., h_n\in \h, m_1, ..., m_n\in \Z_+,\\ m_1+\cdots + m_n + \frac 1 2 \langle \gamma+\beta, \gamma+\beta\rangle = \frac 1 2 \langle \alpha_i, \alpha_i\rangle}} h_1(-m_1)\cdots h_n(-m_n)\iota(e_{\gamma+\beta}). 
    \end{align*}
    From $F(h(0)\iota(e_{\alpha_i})) = h(0)F(\iota(e_{\alpha_i})$ (note that $F(h(-1)\one)= 0$) for any $h\in \h$, we see that 
    \begin{align*}
        &\langle h, \alpha_i\rangle \sum_{\beta\in L_0,  \langle \gamma+\beta, \gamma+\beta\rangle = \langle \alpha_i, \alpha_i\rangle} y_\beta \iota(e_{\gamma+\beta}) \\
        & + \langle h, \alpha_i\rangle \sum_{\substack{\beta\in L_0,  h_1, ..., h_n\in \h, m_1, ..., m_n\in \Z_+,\\ m_1+\cdots + m_n + \frac 1 2 \langle \gamma+\beta, \gamma+\beta\rangle = \frac 1 2 \langle \alpha_i, \alpha_i\rangle}} h_1(-m_1)\cdots h_n(-m_n)\iota(e_{\gamma+\beta})\\
        =&  \sum_{\beta\in L_0,  \langle \gamma+\beta, \gamma+\beta\rangle = \langle \alpha_i, \alpha_i\rangle} y_\beta \langle h, \gamma+\beta\rangle \iota(e_{\gamma+\beta}) \\
        & + \sum_{\substack{\beta\in L_0,  h_1, ..., h_n\in \h, m_1, ..., m_n\in \Z_+,\\ m_1+\cdots + m_n + \frac 1 2 \langle \gamma+\beta, \gamma+\beta\rangle = \frac 1 2 \langle \alpha_i, \alpha_i\rangle}}\langle h, \gamma+\beta\rangle  h_1(-m_1)\cdots h_n(-m_n)\iota(e_{\gamma+\beta}). 
    \end{align*}
\end{enumerate}
Note that if $\langle h, \alpha_i\rangle = \langle h, \gamma+\beta\rangle$ for every $h\in h$, then $\alpha_i = \gamma+\beta$, which is impossible since $\gamma\notin L_0$ while $\alpha_i, \beta\in L_0$. Thus to make the equality, it is necessary that $y_\beta=0$ and all the $h_1(-m_1)\cdots h_n(-m_n)\iota(e_{\gamma+\beta}) = 0$. So $F_1(\iota(e_{\alpha_i}))=0$ for every $i = 1, ..., r$. Thus 
$$F(v) = \sum_{\alpha\in L_0, \langle \gamma+\alpha, \gamma+\alpha\rangle = 2} t_\alpha \iota(e_{\gamma+\alpha})_0 v$$
is a zero-mode derivation. 
\end{proof}

\begin{rema}
For the lattice VOA case, the bimodules for Zhu's algebra does not help at all (cf. Remark \ref{Rmk-1-23}). Indeed, it follows from the 
$$V_{\mu+L_0} \boxtimes V_{\nu+L_0} = V_{\mu+\nu+L_0}$$
among modules assoicated with $\mu, \nu\in L$ that $A_0(V_{\mu+L_0}) = 0$ for any $\mu\notin L_0$. 
\end{rema}

\subsection{The module $V_{\gamma+L_0}$ with $\N$-grading} In this subsection we consider $W=V_{\gamma+\alpha}$ with shifted grading $\d_W = L(0) - \langle \gamma, \gamma\rangle / 2$. In this case, $H^1(V, W)$ consists of derivations satisfying 
$$F(L(0)v) = \left(L(0) - \frac{\langle \gamma, \gamma\rangle }{2} \right)F(v). $$
Similarly, with this grading, the choice of derivations is different. Thus, $H^1(V, W)$ and $Z^1(V, W)$ are different from those with $L(0)$-gradings. 

\begin{thm}\label{lattice-canonical-N-grading}
Let $V=V_{L_0}$ with $L_0$ a positive definite even lattice. Let $W = V_{\gamma+L_0}$ for some $\gamma\notin L_0$ with the grading operator $\d_W = L(0) - \langle \gamma, \gamma\rangle / 2$. Then $H^1(V, W) = Z^1(V, W)$. 
\end{thm}

\begin{proof}
It is clear that $\iota(e_\gamma)$ is a lowest weight vector of $W$ that generates $V_{\gamma+L_0}$. For $h\in \h$, let 
\begin{align}
    F(h(-1)\one) = \sum_{\substack{\alpha\in L_0,\\ \langle \gamma+\alpha, \gamma+\alpha\rangle = \langle \gamma, \gamma\rangle + 2}} x_\alpha(h) \iota(e_{\gamma+\alpha}) + \sum_{\substack{\beta\in L_0, \\ \langle \gamma+\beta, \gamma+\beta\rangle = \langle \gamma, \gamma\rangle}} \sum_{i=1}^r y_{i, \beta}(h)\alpha_i(-1)\iota(e_{\gamma+\beta}), \label{lattice-N-grading-1}
\end{align}
where $x_\alpha, y_{i,\beta}\in \h^*$ for each choice of $\alpha, i, \beta$. 
From 
$$F(h_1(0)h(-1)\one) = h_1(0)F(h(-1)\one) + F(h_1(-1)\one)_0h(-1)\one, $$
we see that
\begin{align}
    0 &=  \sum_{\substack{\alpha\in L_0,\\ \langle \gamma+\alpha, \gamma+\alpha\rangle = \langle \gamma, \gamma\rangle + 2}} \langle h_1, \gamma+\alpha\rangle x_\alpha(h) \iota(e_{\gamma+\alpha}) + \sum_{\substack{\beta\in L_0, \\ \langle \gamma+\beta, \gamma+\beta\rangle = \langle \gamma, \gamma\rangle}} \sum_{i=1}^r \langle h_1, \gamma+\beta\rangle  y_{i, \beta}(h)\alpha_i(-1)\iota(e_{\gamma+\beta})\label{5-6-line-1}\\
    & \quad - \sum_{\substack{\alpha\in L_0,\\ \langle \gamma+\alpha, \gamma+\alpha\rangle = \langle \gamma, \gamma\rangle + 2}} \langle h, \gamma+\alpha\rangle x_\alpha(h_1) \iota(e_{\gamma+\alpha}) - \sum_{\substack{\beta\in L_0, \\ \langle \gamma+\beta, \gamma+\beta\rangle = \langle \gamma, \gamma\rangle}} \sum_{i=1}^r \langle h, \gamma+\beta\rangle y_{i, \beta}(h_1)\alpha_i(-1)\iota(e_{\gamma+\beta})\label{5-6-line-2}\\
    & \quad + \sum_{\substack{\beta\in L_0, \\ \langle \gamma+\beta, \gamma+\beta\rangle = \langle \gamma, \gamma\rangle}} \sum_{i=1}^r \langle h, \alpha_i\rangle y_{i, \beta}(h_1)(\gamma+\beta)(-1)\iota(e_{\gamma+\beta}) \label{5-6-line-3}
\end{align}
Comparing the coefficients of $\iota(e_{\gamma+\alpha})$ for each $\alpha\in L_0$ with $\langle \gamma+\alpha, \gamma+\alpha\rangle = \langle \alpha, \alpha\rangle + 2$, we see that 
$$\langle h_1, \gamma+\alpha\rangle x_\alpha(h) = \langle h, \gamma+\alpha\rangle x_\alpha(h_1). $$
We similarly set $t_\alpha = x_\alpha(h)/\langle h, \gamma+\alpha\rangle$. 

To compare the coefficients for $\alpha_i(-1)\iota(e_{\gamma+\beta})$ for $i=1, ..., r, \beta\in L_0, \langle \gamma+\beta, \gamma+\beta\rangle = \langle \gamma, \gamma\rangle$, we first rewrite 
$$\gamma+\beta= \sum_{j=1}^r \langle \gamma+\beta, \alpha_i^\vee\rangle \alpha_i, $$
where $\alpha_1^\vee, ..., \alpha_r^\vee\in \h$ satisfy $\langle \alpha_i^\vee, \alpha_j\rangle = \delta_{ij}, i, j = 1,..., r$. Then rewrite the last term (\ref{5-6-line-3}) as
\begin{align*}
    & \sum_{\substack{\beta\in L_0, \\ \langle \gamma+\beta, \gamma+\beta\rangle = \langle \gamma, \gamma\rangle}} \sum_{i=1}^r \langle h, \alpha_i\rangle y_{i, \beta}(h_1)\sum_{j=1}^r \langle \gamma+\beta, \alpha_j^\vee \rangle \alpha_j(-1)\iota(e_{\gamma+\beta})\\
    & = \sum_{\substack{\beta\in L_0, \\ \langle \gamma+\beta, \gamma+\beta\rangle = \langle \gamma, \gamma\rangle}}\sum_{i=1}^r \left(\sum_{j=1}^r \langle h, \alpha_j\rangle y_{j, \beta}(h_1) \langle \gamma+\beta, \alpha_i^\vee \rangle\right) \alpha_i(-1)\iota(e_{\gamma+\beta})
\end{align*}
to read the coefficient for $\alpha_i(-1)\iota(e_{\gamma+\beta})$, which yields
$$\langle h_1, \gamma+\beta\rangle y_{i,\beta}(h) = \langle h, \gamma+\beta\rangle y_{i,\beta}(h_1) - \sum_{j=1}^r \langle h, \alpha_j\rangle \langle \gamma+\beta, \alpha_i^\vee\rangle y_{j,\beta}(h_1)$$
Pick $h_1\in \h$ such that $\langle h_1, \gamma+\beta\rangle \neq 0$ for every index $\beta$ involved in this sum. Let $u_{j, \beta} = y_{j, \beta}(h_1)/\langle h_1, \gamma+\beta\rangle$ for each $j=1, ..., r, \beta\in L_0, \langle \gamma+\beta, \gamma+\beta\rangle = \langle \gamma, \gamma\rangle$. Then we have
$$y_{i, \beta}(h) = \langle h, \gamma+\beta\rangle u_{i, \beta} - \sum_{j=1}^r \langle h, \alpha_j\rangle \langle \gamma+\beta,\alpha_i^\vee\rangle u_{j, \beta}. $$
With the choice of $t_\alpha$ and $u_{j, \beta}$, we now consider 
$$F_1(v) = F(v) + \sum_{\substack{\alpha\in L_0 \\ \langle \gamma+\alpha, \gamma+\alpha\rangle = \langle \alpha, \alpha\rangle + 2}} t_\alpha \iota(e_{\gamma+\alpha})_0 v + \sum_{\substack{\beta\in L_0 \\ \langle \gamma+\beta, \gamma+\beta\rangle = \langle \gamma, \gamma\rangle}}\sum_{i=1}^r u_{i, \beta}\alpha_i(-1)\iota(e_{\gamma+\beta})_0 v. $$
Since we modified $F$ with zero-mode derivations, it is clear that $F_1$ is also a derivation. The following computation shows that $F_1(h(-1)\one) = 0$ for every $h\in \h$. Indeed, 
\begin{align*}
     &\quad F_1(h(-1)\one)\\
     &= F(h(-1)\one) - \sum_{\substack{\alpha\in L_0 \\ \langle \gamma+\alpha, \gamma+\alpha\rangle = \langle \alpha, \alpha\rangle + 2}} t_\alpha \langle h, \gamma+\alpha\rangle \iota(e_{\gamma+\alpha}) \\
    & \quad - \sum_{\substack{\beta\in L_0 \\ \langle \gamma+\beta, \gamma+\beta\rangle = \langle \gamma, \gamma\rangle}}\sum_{i=1}^r u_{i, \beta}\left(\langle h, \gamma+\beta\rangle \alpha_i(-1)\iota(e_{\gamma+\beta})- \sum_{j=1}^r \langle h, \alpha_i\rangle \langle \gamma+\beta, \alpha_j^\vee\rangle \alpha_j(-1)\iota(e_{\gamma+\beta})\right)\\
    &= F(h(-1)\one) - \sum_{\substack{\alpha\in L_0 \\ \langle \gamma+\alpha, \gamma+\alpha\rangle = \langle \alpha, \alpha\rangle + 2}} x_\alpha(h) \iota(e_{\gamma+\alpha}) \\
    & \quad - \sum_{\substack{\beta\in L_0 \\ \langle \gamma+\beta, \gamma+\beta\rangle = \langle \gamma, \gamma\rangle}} \left(\sum_{i=1}^r u_{i, \beta}\langle h, \gamma+\beta\rangle \alpha_i(-1)\iota(e_{\gamma+\beta})- \sum_{i=1}^r\sum_{j=1}^r \langle h, \alpha_j\rangle \langle \gamma+\beta, \alpha_i^\vee\rangle \alpha_i(-1)\iota(e_{\gamma+\beta}) \right)\\
    &= F(h(-1)\one) - \sum_{\substack{\alpha\in L_0 \\ \langle \gamma+\alpha, \gamma+\alpha\rangle = \langle \alpha, \alpha\rangle + 2}} x_\alpha(h) \iota(e_{\gamma+\alpha}) - \sum_{\substack{\beta\in L_0 \\ \langle \gamma+\beta, \gamma+\beta\rangle = \langle \gamma, \gamma\rangle}}\sum_{i=1}^r y_{i, \beta}(h) \alpha_i\iota(e_{\gamma+\beta}) \\
    &=0
\end{align*}
Let 
\begin{align*}
    & F_1(\iota(e_{\alpha_i})) \\
    & = \sum_{\substack{\alpha\in L_0, \\ \langle \gamma+\alpha,\gamma+\alpha\rangle \\
    = \langle \alpha_i, \alpha_i\rangle+\langle \gamma, \gamma\rangle }} z_\alpha \iota(e_{\gamma+\alpha}) \\
    & +   \sum_{\substack{\beta\in L_0, s_1, ..., s_r \in \N,\\ k_1^{(1)}, ..., k_{s_1}^{(1)}, ..., k_1^{(r)}, ..., k_{s_r}^{(r)}\in \N\\ \sum_{p=1}^r \sum_{j=1}^{s_i} j\cdot k_j^{(p)} +  \frac 1 2 \langle \gamma+\beta, \gamma+\beta\rangle \\
    = \frac 1 2 \langle \alpha_i, \alpha_i\rangle  + \frac 1 2 \langle \gamma, \gamma\rangle }}& & u_{k_1^{(1)}, ..., k_{s_1}^{(1)}, ..., k_1^{(r)}, ...,  k_{s_r}^{(r)}, \beta} \\
    & &\cdot & \alpha_1(-1)^{k_1^{(1)}}\cdots \alpha_1(-s_1)^{k_{s_1}^{(1)}}\cdots \alpha_r(-1)^{k_1^{(r)}}\cdots \alpha_r(-s_r)^{k_{s_r}^{(r)}}\iota(e_{\gamma+\beta}). 
\end{align*}
Since $F_1(h(-1)\one) = 0$, it is clear that 
$$F(h(0) \iota(e_{\alpha_i})) = h(0)F(\iota(e_{\alpha_i})), $$
i.e., 
\begin{align*}
    & \quad \sum_{\substack{\alpha\in L_0, \\ \langle \gamma+\alpha,\gamma+\alpha\rangle \\
    = \langle \alpha_i, \alpha_i\rangle+\langle \gamma, \gamma\rangle }} \langle h, \alpha_i\rangle z_\alpha \iota(e_{\gamma+\alpha}) \\
    & +   \sum_{\substack{\beta\in L_0, s_1, ..., s_r \in \N,\\ k_1^{(1)}, ..., k_{s_1}^{(1)}, ..., k_1^{(r)}, ..., k_{s_r}^{(r)}\in \N\\ \sum_{p=1}^r \sum_{j=1}^{s_i} j\cdot k_j^{(p)} +  \frac 1 2 \langle \gamma+\beta, \gamma+\beta\rangle \\
    = \frac 1 2 \langle \alpha_i, \alpha_i\rangle  + \frac 1 2 \langle \gamma, \gamma\rangle }}& & \langle h, \alpha_i\rangle u_{k_1^{(1)}, ..., k_{s_1}^{(1)}, ..., k_1^{(r)}, ...,  k_{s_r}^{(r)}, \beta} \\
    & &\cdot & \alpha_1(-1)^{k_1^{(1)}}\cdots \alpha_1(-s_1)^{k_{s_1}^{(1)}}\cdots \alpha_r(-1)^{k_1^{(r)}}\cdots \alpha_r(-s_r)^{k_{s_r}^{(r)}}\iota(e_{\gamma+\beta}). \\
    & =  \sum_{\substack{\alpha\in L_0, \\ \langle \gamma+\alpha,\gamma+\alpha\rangle \\
    = \langle \alpha_i, \alpha_i\rangle+\langle \gamma, \gamma\rangle }} \langle h, \gamma+\alpha\rangle z_\alpha \iota(e_{\gamma+\alpha}) \\
    &    \sum_{\substack{\beta\in L_0, s_1, ..., s_r \in \N,\\ k_1^{(1)}, ..., k_{s_1}^{(1)}, ..., k_1^{(r)}, ..., k_{s_r}^{(r)}\in \N\\ \sum_{p=1}^r \sum_{j=1}^{s_i} j\cdot k_j^{(p)} +  \frac 1 2 \langle \gamma+\beta, \gamma+\beta\rangle \\
    = \frac 1 2 \langle \alpha_i, \alpha_i\rangle  + \frac 1 2 \langle \gamma, \gamma\rangle }}& & \langle h, \gamma+\beta\rangle u_{k_1^{(1)}, ..., k_{s_1}^{(1)}, ..., k_1^{(r)}, ...,  k_{s_r}^{(r)}, \beta} \\
    & &\cdot & \alpha_1(-1)^{k_1^{(1)}}\cdots \alpha_1(-s_1)^{k_{s_1}^{(1)}}\cdots \alpha_r(-1)^{k_1^{(r)}}\cdots \alpha_r(-s_r)^{k_{s_r}^{(r)}}\iota(e_{\gamma+\beta}). 
\end{align*}
Thus for every choice of indices $\alpha, \beta$ involved in the sum, 
\begin{align}
    \langle h, \alpha_i \rangle z_\alpha& = \langle h, \gamma+\alpha\rangle z_\alpha, \label{lattice-N-grading-2}\\
    \langle h, \alpha_i \rangle u_{k_1^{(1)}, ..., k_{s_1}^{(1)}, ..., k_1^{(r)}, ...,  k_{s_r}^{(r)}, \beta}&= \langle h, \gamma+\beta\rangle u_{k_1^{(1)}, ..., k_{s_1}^{(1)}, ..., k_1^{(r)}, ...,  k_{s_r}^{(r)}, \beta} \label{lattice-N-grading-3}
\end{align}
holds for arbitrary $h\in \h$. Since $\gamma\notin L_0$, this is possible only when every $z_\alpha = 0$ and every $u_{k_1^{(1)}, ..., k_{s_1}^{(1)}, ..., k_1^{(r)}, ...,  k_{s_r}^{(r)}, \beta} = 0$. Thus 
$$F_1(\iota(e_{\alpha_i})) = 0, i = 1, ..., r.$$
That is to say, $F_1(v) = 0$. Thus 
$$F(v) = - \sum_{\substack{\alpha\in L_0 \\ \langle \gamma+\alpha, \gamma+\alpha\rangle = \langle \alpha, \alpha\rangle + 2}} t_\alpha \iota(e_{\gamma+\alpha})_0 v - \sum_{\substack{\beta\in L_0 \\ \langle \gamma+\beta, \gamma+\beta\rangle = \langle \gamma, \gamma\rangle}}\sum_{i=1}^r u_{i, \beta}\alpha_i(-1)\iota(e_{\gamma+\beta})_0 v$$
is a zero-mode derivation. 
\end{proof}

\begin{thm}
Let $V=V_{L_0}$ where $L_0$ a positive definite even lattice. Let $W = V_{\gamma+L_0}$ for some $\gamma\notin L_0$ with arbitrary $\N$-grading. Then $H^1(V, W) = Z^1(V, W)$. 
\end{thm}

\begin{proof}
Let $\Xi$ be the lowest weight of $W$. The case $\Xi=0$ is solved in Theorem \ref{lattice-canonical-N-grading}. For $\Xi\geq 1$, the process is a trivial modification from that in Theorem \ref{lattice-canonical-N-grading}. We shall not repeat the details but just give a sketch. 
\begin{enumerate}
    \item When $\Xi \geq 2$, $F(h(-1)\one)$ is automatically zero. We take $F_1 = F$. When $\Xi=1$, the expression of $F(h(-1)\one)$ in (\ref{lattice-N-grading-1}) involves only the first sum. Using the same arguments in Theorem \ref{lattice-L(0)-thm}, we find an element $w_{(1)}\in W_{(1)}$, so that the derivation $F_1(v) = F(v) - (w_{(1)})_0 v$ satisfies $F_1(h(-1)\one)= 0, h\in \h$. 
    \item For every basis element $\alpha_i$ of the lattice $L_0$, the condition  $F_1(h(0)\iota(e_{\alpha_i}) = h(0) \iota(e_{\alpha_i})$ implies (\ref{lattice-N-grading-2}) and (\ref{lattice-N-grading-3}) for any $h\in \h$. If $\gamma\notin L_0$, it follows similarly that $F_1(\iota(e_{\alpha_i})) = 0$. If $\gamma\in L_0$, in which case we should take $\gamma = 0$, then (\ref{lattice-N-grading-2}) and (\ref{lattice-N-grading-3}) implies that $\alpha = \alpha_i$ and $\beta= \alpha_i$. But from the grading-preserving condition of $F$, we should have $\langle \alpha, \alpha\rangle + 2\Xi = \langle \alpha_i, \alpha_i\rangle$ and $\langle \beta, \beta\rangle + \sum_{p=1}^r \sum_{j=1}^{s_i}j\cdot k_j^{(p)} + \Xi = \langle \alpha_i, \alpha_i\rangle$, neither of which is possible. So both $z_\alpha$ in (\ref{lattice-N-grading-2}) and $u_{k_1^{(1)},..., k_{s_1}^{(1)}, ..., k_1^{(r)}, ..., k_{s_r}^{(r)}, \beta}$ in (\ref{lattice-N-grading-3}) are zero. The same conclusion $F_1(\iota(e_{\alpha_i}))= 0$ holds. 
\end{enumerate}

\end{proof}

\section{Summary and Remarks}

In Sections \ref{Section-2}, \ref{Section-3} and \ref{Section-4}, we computed the first cohomology $H^1(V, W)$ for the three classes of VOAs $V$ and for any irreducible $V$-module $W$ with arbitrary $\N$-grading. Since every $\N$-graded $V$-module is a direct sum of irreducible submodules, we conclude the following theorem: 

\begin{thm}
Let $V$ be a VOA and $W$ be a $V$-module, then $H^1(V, W) = Z^1(V, W)$ if 
\begin{enumerate}
    \item $V = L_{\hat\g}(l, 0)$ where $\g$ is a simple Lie algebra, $l\in \Z_+$, $W$ is any $\N$-graded $V$-module;
    \item $V = L(c, 0)$ where $c = c_{pq}$ as in (\ref{Formula-5}), $W$ is any $\N$-graded $V$-module;
    \item $V=V_{L_0}$ where $L_0$ is a positive definite even lattice, $W$ is any $\N$-graded $V$-module.
\end{enumerate}
\end{thm}



In case $V$ is the Virasoro VOA $L(c,0)$ with $c\neq c_{pq}$, $c_{pq}$ as in (\ref{Formula-5}), Proposition \ref{L(-1)w-der} shows that there exists an $\N$-graded module $W$ such that $H^1(V, W)\neq Z^1(V, W)$.

In case $V$ is the Virasoro VOA $L(c,0)$ with $c=c_{pq}$ as in (\ref{Formula-5}) and $V$ admits negative energy representations, we also proved that $H^1(V, W) = Z^1(V, W)$ for every irreducible $L(0)$-graded $V$-module whose lowest weight is greater or equal to $-3$. Since every $L(0)$-graded $V$-module is a direct sum of irreducibles, we also managed to prove that

\begin{thm}
Let $V = L(c, 0)$ be the Virasoro VOA where $c = c_{pq}$ as in (\ref{Formula-5}). Then for every $L(0)$-graded $V$-module $W$ whose lowest weight is greater or equal to $-3$, $H^1(V, W) = Z^1(V, W)$. 
\end{thm}

\begin{rema}
In \cite{HQ-Coh-reduct}, we conjectured that $H^1(V, W) = Z^1(V, W)$ for any $V$-bimodule $W$. While this paper provides concrete evidence for us to believe that the conjectures hold in general, to fully prove the conjecture we need to classify all irreducible $V$-bimodules. For some examples of $V$ (affine $sl(2)$, minimal model Virasoro), an irreducible $V$-module $W$ admits a unique $V$-bimodule structure where $Y_W^R = Y_{WV}^W$. So for these examples, the classification of bimodules might not be very difficult and shall be attempted in the future. 
\end{rema}

\begin{rema}
The arguments so far did not involve any discussions of $C_2$-cofiniteness, which is key to the rigidity of the module categories. The relationship between the cofiniteness conditions and the cohomology theory should be carefully investigated in future work. 
\end{rema}

\noindent {\small \sc Pacific Institute of Mathematical Science | University Of Manitoba\\ 451 Machray Hall, 186 Dysart Road, Winnipeg, MB R3T 2N2, Canada}

\noindent {\em E-mail address}: fei.qi@umanitoba.ca

The author states that there is no conflict of interest. 

\end{document}